\documentclass[11pt,reqno]{amsart}
\usepackage{amsmath}

\usepackage{amsfonts}
\usepackage{amssymb}
\usepackage[dvips]{graphicx}
\usepackage{color}
\usepackage[pagewise]{lineno}
\usepackage{url}
\theoremstyle{plain}
\newtheorem{athm}{Theorem}

\newtheorem{theorem}{Theorem}[section]
\newtheorem{proposition}[theorem]{Proposition}
\newtheorem{corollary}[theorem]{Corollary}
\newtheorem{lemma}[theorem]{Lemma}

\theoremstyle{definition}

\newtheorem{example}{Example}[section]
\newtheorem{definition}[theorem]{Definition}

\newtheorem{remark}[theorem]{Remark}

\DeclareMathOperator{\esssup}{ess \ sup}

\DeclareMathOperator{\diam}{diam}
\DeclareMathOperator{\lip}{Lip}

\input{tcilatex}

\begin{document}

\begin{abstract}
In this paper, we establish the quasi-compactness of the transfer operator associated with skew product systems that are semi-conjugate to piecewise convex maps with a countably infinite number of branches. These non-invertible skew products admit discontinuities, with the critical set confined to a countable collection of fibers. Furthermore, we demonstrate that such systems possess an invariant measure whose disintegration along the fibers exhibits bounded variation; a concept introduced and developed in this work.
\end{abstract}

\title[Quasi C. for systems semi-conjugate to p. convex maps]{Quasi-compactness and statistical properties for discontinuous systems semi-conjugate to piecewise convex maps with countably many branches}

\author[Rafael Lucena]{Rafael Lucena}

\date{\today }

\dedicatory{In loving memory of my dear friend Fellyppe Carlos Santos de Lima.}

\address[Rafael Lucena]{Universidade Federal de Alagoas, Instituto de Matemática - UFAL, Av. Lourival Melo Mota, S/N
	Tabuleiro dos Martins, Maceio - AL, 57072-900, Brasil}
\email{rafael.lucena@im.ufal.br}
\urladdr{www.im.ufal.br/professor/rafaellucena}
\maketitle

%\thanks{2010 \emph{Mathematics Subject Classification}: Primary 37D20; Secondary 37C20.}

\section{Introduction}

In the study of deterministic dynamical systems, a central goal is to understand long‐term statistical properties such as mixing rates, decay of correlations, and the existence and uniqueness of invariant measures. A particularly powerful approach to these issues is the analysis of the associated transfer (or Frobenius-Perron) operators, whose spectral properties frequently yield precise insights into the system’s ergodic and statistical behavior. In particular, the presence of a spectral gap in the transfer operator is a strong indicator of rapid (exponential) decay of correlations for an appropriate class of observables.

Building on the rich legacy of one‐dimensional systems, most notably, the seminal work of Lasota and Yorke on piecewise monotonic maps with finitely many branches, we consider in this paper a class of two‐dimensional skew product systems of the form $F(x,y)=(f(x),G(x,y))$, where the horizontal component $f$ is a piecewise convex map with a countably infinite number of branches, and the vertical component $G$ contracts almost every vertical fiber while admitting discontinuities. More precisely, the discontinuities of $G$ are confined to a countable union of vertical fibers. This framework not only generalizes classical results for one‐dimensional maps with finitely many branches but also introduces additional challenges due to the interplay between expansion in the horizontal direction and contraction along the fibers.

Our approach is inspired by the works in \cite{RG2} and \cite{GLu}. In the first, P. Gora and A. Rajput demonstrated that the Frobenius-Perron operator for piecewise convex maps with countably many branches is quasi-compact, alongside establishing several other significant properties of its invariant measure. In the second work, the authors proved the existence of a spectral gap for the transfer operator associated with Lorenz-like systems, introducing innovative ideas for constructing anisotropic spaces that are well-suited to handling skew products. In essence, our approach shows that even non-invertible systems, which feature a highly intricate discontinuity structure, can exhibit robust ergodic and statistical properties. Building on two previous works, in this paper we present a study of skew-products that contract almost every fiber in the sense of Lebesgue and that are semi-conjugate to piecewise convex maps with an infinite, countable number of branches, admitting infinitely many discontinuity points (possible non-countable), provided these are confined to a countable set of fibers. We prove that the transfer operator associated with these systems exhibits a spectral gap in an appropriate vector space of signed measures. Consequently, we demonstrate that the map has a unique invariant measure in the considered space and that this measure is regular, in the sense that its disintegration along the fibers exhibits bounded variation (BV). As a result of this regularity, we further show that these systems exhibit exponential decay of correlations for Lipschitz observables.

In Section \ref{kjdfkjdsfkj}, we present the basic assumptions of the main system and some key results concerning it. In Section \ref{kjdfkjdsfkj}, we also present several examples to illustrate the applicability of the theory developed in this paper. Next, in Section \ref{utrietyyighgdf}, we construct vector spaces of signed measures that are suitable for studying the transfer operator of $F$; on these spaces, the operator exhibits a spectral gap. Subsequently, in Section \ref{xbcvhgsafd}, we define a vector space of signed measures with greater regularity than those introduced in Section \ref{utrietyyighgdf}. We study the action of the transfer operator on an invariant subset and demonstrate that this action satisfies a Lasota–Yorke inequality. Consequently, we can employ the convergence properties obtained in Section \ref{utrietyyighgdf} to conclude that the unique $F$-invariant measure in these spaces possesses a regular disintegration along the fibers. Finally, in Section \ref{dfjgsghdfjasdf}, we apply the previous results to show that the system in question exhibits exponential decay of correlations over the space of Lipschitz observables.

This paper presents four main results, namely Theorems A, B, C, and D. In Theorem A, we prove that $F$ has a unique invariant measure in a vector space of signed measures, with respect to which the transfer operator associated with $F$ exhibits a spectral gap. This spectral gap property is the key result of the paper and is stated in Theorem B. In Theorem C, we show that the $F$-invariant measure from Theorem A possesses an even stronger property: its disintegration has bounded variation, a concept introduced in Section \ref{xbcvhgsafd}. Finally, in Theorem D, we present an application of Theorems A, B, and the additional regularity established in Theorem C, demonstrating that the system $F$ exhibits exponential decay of correlations for the class of Lipschitz observables.

\textbf{Acknowledgments} 

I extend my heartfelt thanks to my beloved wife, Cibelle, for her unconditional support. I am also deeply grateful to my friend, Davi Lima, for his consistent availability and willingness to discuss mathematical ideas with me. Finally, I would like to thank the referee for their valuable contributions to the improvement of this article.

This work was partially supported by  
\begin{enumerate}
	\item[(1)] FAPEAL (Alagoas-Brazil) Grant E:60030.0000002330/2022; 
	\item[(2)] CNPq (Brazil) Grant Alagoas Din\^amica: 409198/2021-8.
\end{enumerate}

\section{Preliminary Framework and Examples}\label{kjdfkjdsfkj}

Throughout this article, we use the following notation: $I=[0,1]$, $\mathcal{B}$ denotes the Borel $\sigma$ algebra of $I$ and $m$ is the normalized Lebesgue measure on $\mathcal{B}$.  The Euclidean distance on $I$ is denoted by $d_1$ and $K$ represents a compact metric space with metric $d_2$. With regard to \( d_2 \), we assume \( \mathrm{diam}(K) = 1 \) to avoid unnecessary multiplicative constants. Finally, we define $\Sigma: = I \times K$ and consider it a metric space endowed with the product metric $d_1 + d_2$.

In this section, we establish the hypotheses for the dynamical system under study:$$F: \Sigma \longrightarrow \Sigma, \ F=(f(x), G(x,y)),$$where $f: I \longrightarrow I$ and $G: \Sigma \longrightarrow K$.

We first present some results from \cite{RG2} concerning $f$ and derive a few key lemmas that will be useful in the analysis of the main system, $F$. Additionally, we provide important examples illustrating how the theory developed in this section can be applied.

\subsection{Hypothesis on the basis map $f$ and Spectral Gap}

\subsubsection{Piecewise convex map with countably infinite number of branches}

\begin{definition} Let $I=[0,1]$ and let $\mathcal {P}=\{I_i=(a_i,b_i)\}_{i=1}^\infty$ be a  family of open disjoint subintervals of $I$ such that $m\left(I\setminus\bigcup_{i\ge 1}I_i\right)=0$.
	We say that a transformation $f:\bigcup_{i\ge 1}I_i \longrightarrow I$ is called a \emph{piecewise convex map with countably many branches} on the partition $\mathcal {P}$ if it satisfies the following:\\
	(1) For $i=1,2,3..., f_i=f_{|I_i}$ is an increasing convex differentiable function with $\lim_{x\to a_i^+}f_i(x)=0$.
	Define $f_i(a_i)=0$ and $f_i(b_i)=\lim_{x\to b_i^-}f_i(x)$. The values $f_i'(a_i)$ are also defined by continuity;\\
	(2) The derivative of $f$ satisfies $$\sum_{i\ge 1}\frac 1{f_i'(a_i)} < +\infty;$$\\
	(3) If $x=0$ is not a limit point of the partition points, then we have $f'(0)=1/\beta>1$, for some $0<\beta<1$.
	By $\mathcal T$ we will denote the set of maps satisfying conditions (1)-(3). 
	\label{def1}
\end{definition}

 Define $\mathcal P^{(n)}:=\mathcal P \bigvee f^{-1}(\mathcal P)\bigvee\dots\bigvee f^{n-1}(\mathcal P) $. We denote the branches of $f^n$ by $f^{(n)}_i$. Then, $\mathcal P^{(n)}=\left\{I_i^{(n)}=\left(a_i^{(n)},b_i^{(n)}\right)\right\}_{i=1}^\infty$ is a countably infinite family of open disjoint subintervals of $I$ corresponding to $f^n$. We have the following results:

\begin{theorem}
Let $\mathcal{P}$ be a partition for $f$ and $\mathcal{P}^{(n)}$ denote the partition for $f^n$. If $f \in \mathcal{T}$ then $f^n \in \mathcal{T}$ as well, i.e.,\\
\textbf{(a)} If $f\in \mathcal{T}$ then $f^n$  is piecewise increasing on $\mathcal{P}^{(n)}$.\\
\textbf{(b)} $f^n$  is piecewise convex on $\mathcal{P}^{(n)}$.\\
\textbf{(c)}  $f^n$  is piecewise differentiable on $\mathcal{P}^{(n)}$.\\
\textbf{(d)} $\lim_{x \rightarrow \left(a_i^{(n)}\right)^+} f_i^{(n)}(x)=0$ for $f^n$ on $\mathcal{P}^{(n)}$.\\
\textbf{(e)}The condition $(2)$ holds for $f^n$. i.e.,
\begin{equation*}
\sum_{i\ge 1}\frac {1}{\left(f^{(n)}_i\right)'\left(a_i^{(n)}\right)} < +\infty.
\end{equation*}
\textbf{(f)} If $x=0$ is not a limit point of the partition points and condition $(3)$ holds for $\tau$  then it holds for $\tau^n$.
\end{theorem}  

\begin{proof}
	See \cite{RG2}.
\end{proof}

\subsubsection{Piecewise expanding map with countable number of branches}

\begin{definition} \label{def2}Let $I=[0,1]$ and let $\mathcal P=\{I_i=(a_i,b_i)\}_{i=1}^\infty$ be a countably infinite family of open disjoint subintervals of $I$ such that Lebesgue measure of $I\setminus\bigcup_{i\ge 1}I_i$ is zero. Let $f$ be a map from $\displaystyle \bigcup_{i\geq 1} I_i$ into the interval $I$, such that for each $i\geq 1$, $f_{|I_i}$ extends to a homeomorphism $f_i$ of $[a_i,b_i]$ onto its image.\\ Let
	\begin{equation*}
		g(x) = 
		\begin{cases}
			\frac{1}{|f_i'(x)|}, & \text{for } x\in I_i, i=1,2,\dots \\
			0, & \text{elsewhere}
		\end{cases}.
	\end{equation*}
	We assume $\sup_{x\in I} |g(x)| \le \beta <1$. Then, we say $f$ is a piecewise expanding map with countably many branches and denote this class by $\mathcal  T_E$.
\end{definition}

The following lemmas were established in \cite{RG2}, and their proofs are omitted here.
\begin{lemma}\label{lemma1} If $f\in \mathcal{T}$ in the sense of Definition \ref{def1}, then some iterate of $f^n\in \mathcal T_E$ in the sense of Definition \ref{def2}.
\end{lemma}

\begin{lemma}\label{existencef} Let $f \in \mathcal{T}$. Then there exists a unique normalized absolutely continuous $f$-invariant measure $m_1$. The dynamical system $([0,1],\mathcal{B}, m_1; f)$ is exact and the density $\displaystyle h_1=\frac{dm_1}{dm}$ is bounded and decreasing.
\end{lemma}
\subsection{Hypothesis on the fiber map $G$}

We suppose that $G: \Sigma \longrightarrow K$ satisfies:

\begin{enumerate}
	\item [(H1)] $G$ is uniformly contracting on $m$-almost every vertical fiber $\gamma_x :=\{x\}\times K$: there is $0 \leq \alpha <1$ such that for $m$-a.e. $x\in M$ it holds%
	\begin{equation}
		d_2(G(x,z_{1}),G(x,z_2))\leq \alpha d_2(z_{1},z_{2}), \quad \forall
		z_{1},z_{2}\in K.  \label{contracting1}
	\end{equation}
\end{enumerate}We denote the set of all vertical fibers $\gamma_x$, by $\mathcal{F}^s$: $$\mathcal{F}^s:= \{\gamma _x:=\{ x\}\times K; x \in I \} .$$ When no confusion is possible, the elements of $\mathcal{F}^s$ will be denoted simply by $\gamma$, instead of $\gamma _x$.
\begin{remark}
We note that elements of $\mathcal{F}^s$ i.e., $\gamma_x(=\{x\}\times K)$ for $x \in I$, are naturally identified with their "base point" $x$. For this reason, throughout this article, we will occasionally use the same notation for both, without explicit distinction. In other words, the symbol $\gamma$ may refer either to an element of $\mathcal{F}^s$ or to a point in the interval $I$, depending on the context. For instance, if $\phi:I \longrightarrow \mathbb{R}$ is a real-valued function, the expressions $\phi(x)$ and $\phi(\gamma)$ have the same meaning, as we are implicitly identifying $\gamma$ with $x$.
\end{remark}

\begin{enumerate}
	\item [(H2)] Let $I_1, \cdots, I_{s}, \cdots$ be a partition of $I$ given by definitions \ref{def1} or \ref{def2}. Suppose that for all $s \in \mathbb{N}$ it holds
	\begin{equation}\label{oityy}
		|G_s|_{\lip}:= \sup _y\sup_{x_1, x_2 \in I_s} \dfrac{d_2(G(x_1,y), G(x_2,y))}{d_1(x_1,x_2)}< \infty.
	\end{equation}
\end{enumerate}And denote by $|G|_{\lip}$ the following constant
\begin{equation}\label{jdhfjdh}
	|G|_{\lip} := \max_{s=1, 2, \cdots} \{|G_s|_{\lip}\}.
\end{equation}

\begin{remark}
The condition (H2) implies that $G$ may be discontinuous on the sets $\partial I_i \times K$ for all $i=1,2 \cdots$, where $\partial I_i$ denotes the boundary of $I_i$.
\end{remark}

\begin{remark}\label{uyrytuert}
	In some cases, $G$ can be discontinuous along any countably infinite collection of vertical lines of the form $\{x\} \times K$, $x \in I$.
	
	We illustrate this with two cases. First, suppose that $F=(f,G)$ is such that $f$ satisfies Definition \ref{def2} and $G$ satisfies condition (H2) with respect to the partitions $\mathcal{P}$ and $\mathcal{P}_2$, respectively. Then, $f$ also satisfies Definition \ref{def2} with respect to the refined partition $\mathcal{P} \bigvee \mathcal{P}_2$. Consequently, $G$ satisfies (H2) on $\mathcal{P} \bigvee \mathcal{P}_2$ as well.
	
	The second case is more subtle. Suppose now that $f$, instead of satisfying Definition \ref{def2}, satisfies Definition \ref{def1} with respect to a partition $\mathcal{P}$, and $G$ satisfies (H2) on a different partition $\mathcal{P}_2 \neq \mathcal{P}$. This situation can still be handled, for instance, if condition (H2) is preserved under iteration. That is, if for each $n$, the iterate $F^n=(f^n, G_n)$ is such that $G_n$ satisfies (H2). Thus, by Lemma \ref{lemma1}, there exists an iterate $f^n$ that satisfies Definition \ref{def2} with respect to the partition $\mathcal{P}^{(n)}$. In particular, $f^n$ satisfies Definition \ref{def2} on the partition $(\mathcal{P} \bigvee \mathcal{P}_2)^{(n)}$, and since (H2) is preserved, it follows that $G_n$ satisfies (H2) on the same partition $(\mathcal{P} \bigvee \mathcal{P}_2)^{(n)}$.
	
	An example where this second case may occur is when $f|_P$ is Lipschitz on each $P \in \mathcal{P}$, and the Lipschitz constants of the family $\{f|_P\}_{P \in \mathcal{P}}$ are uniformly bounded. However, such uniform boundedness is incompatible with condition (2) of Definition~\ref{def1} when the partition $\mathcal{P}$ is infinite. Consequently, examples satisfying this scenario can only arise when $\mathcal{P}$ is finite.

	Another admissible situation, where $G$ can be discontinuous along any countably infinite collection of vertical lines of the form $\{x\} \times K$, $x \in I$, occurs when condition (H2) is satisfied by $G_n$ at the iterate $n$ provided by Lemma~\ref{lemma1} (see hypothesis (H3) below).
	
\end{remark}

\begin{enumerate}
	\item [(H3)] There exists an iterate $k \in \mathbb{N}$ such that $F^k=(f^k, G_k)$ satisfies $$\alpha_4:=\alpha^k \esssup \frac{1}{|(f^k)'|} < 1,$$where the essential supremum is taken with respect to $m$. Moreover, $G_k$ satisfies $(H1)$ (with contraction rate $\alpha ^k$) and $(H2)$. 
\end{enumerate}

\begin{definition}\label{closed}
	We say that condition (H2) is \emph{closed} if, for every \( n \geq 1 \), the function \( G_n \) satisfies condition (H2), where \( G_n \) denotes the fiber component associated to the \( n \)-th iterate of the map \( F = (f, G) \), i.e., $F^n=(f^n, G_n)$.
\end{definition}

\begin{remark}
	With this definition, the second case described in Remark~\ref{uyrytuert} occurs if condition (H2) is closed. Moreover, if (H2) is closed  then (H3) is satisfied, as well. 
\end{remark}

Let \( m_1 \) be the \( f \)-invariant measure whose existence is guaranteed by Lemma~\ref{existencef}. Proposition~\ref{kjdhkskjfkjskdjf} below establishes the existence and uniqueness of an \( F \)-invariant measure \( \mu_0 \) that projects onto \( m_1 \). The proof is omitted here and can be found in Theorem~9.4 and Proposition~9.5 of \cite{DR}, where all details are provided. 

Therefore, if \( F \colon \Sigma \to \Sigma \), with \( F = (f, G) \), where \( f \in \mathcal{T} \) and \( G \) satisfies conditions (H1), then there exists a unique \( F \)-invariant measure \( \mu_0 \).

\begin{proposition}\label{kjdhkskjfkjskdjf}
	Let $m_1$ be an $f$-invariant probability. If $F$ satisfies (H1), then there exists an unique measure $\mu_0$ on $M \times K$ such that $\pi_1{_\ast}\mu_0 = m_1$ and for every continuous function $\psi \in C^0 (M \times K)$ it holds 
	\begin{equation*}
		\lim {\int{\inf_{\gamma \times K} \psi \circ F^n }dm_1(\gamma)}= \lim {\int{\sup_{\gamma \times K} \psi \circ F^n}dm_1 (\gamma)}=\int {\psi}d\mu_0. 
	\end{equation*}Moreover, the measure $\mu_0$ is $F$-invariant.
\end{proposition}

\subsection{Examples}\label{dkjfhksjdhfksdf}

\subsubsection{Examples for the base map $f$.}

\begin{example}[Slopes]\label{slopes1}
	Consider $I=[0,1]$ and let $m$ be the Lebesgue measure on $I$. Let $(I_i)_{i \in \mathbb{N}}$ be a family of pairwise disjoint open intervals such that $m\left(I \setminus \bigcup_{i \geq 1} I_i\right) = 0$. Consider $f: \bigcup _{i=1}^\infty I_i  \longrightarrow [0,1]$ such that $f_i:=f|_{I_i}$ is linear with slope $k_i$. Moreover, assume that $\inf_{i \in \mathbb{N}} k_i > 1$ and $\sum_{i=1}^\infty k_i^{-1} < +\infty$. Thus, $f$ satisfies Definition \ref{def2}. In particular, $f$ satisfies corollaries \ref{oirtyuv} and \ref{coro1}, and Proposition \ref{jkkgnhn}. Further details can be found in~\cite{RLK}.
\end{example}

\begin{example}[Slopes]\label{slopes2}
	Let $I = [0,1]$, and let $m$ denote the Lebesgue measure on $I$. Consider a countable family of pairwise disjoint open intervals $(I_i)_{i \in \mathbb{N}}$ such that $m\left(I \setminus \bigcup_{i \geq 1} I_i\right) = 0$. Define a map $f: \bigcup_{i=1}^\infty I_i \to [0,1]$ such that each branch $f_i := f|_{I_i}$ is linear with slope $k_i$ in a way that condition (1) of Definition~\ref{def1} is satisfied.
	
	Assume that $0 < k_i < 1$ for only finitely many indices $i \geq 2$ (excluding $k_1$, $I_1=(0, b_1)$), and denote this finite set of indices by $\mathbb{N}_1$. Moreover, suppose that $\inf_{i \in \mathbb{N}_1^c} k_i > 1$ and that $\sum_{i=1}^\infty k_i^{-1} < +\infty$. 
	
	Under these assumptions, the map $f$ satisfies Definition~\ref{def1}, but not Definition~\ref{def2} at time one. In particular, $f$ satisfies corollaries~\ref{oirtyuv}, ~\ref{coro1}, as well as Proposition~\ref{jkkgnhn} and $f'(x)<1$ for all $x \in  \bigcup_{i \in \mathbb{N}_1} I_i$.
\end{example}

\begin{example}[Gauss Map]	
	\label{gauss}
	Let $\mathcal{P}=(I_i)_{i=1}^\infty$ be the partition of $[0,1]$ where $I_i = (\frac{1}{i+1}, \frac{1}{i})$ for all $i$. Define  $f: \bigcup _{i=1}^\infty I_i : \longrightarrow [0,1]$, by $f(x)=\dfrac{1}{x}-i$ for all $x \in I_i$. Note that $\inf (f^2)' \geq 2$, so it satisfies Definition~\ref{def2}, as discussed in Remark~\ref{remark1}. Moreover, \( f \) admits a mixing probability measure \( m_1 \) absolutely continuous with respect to \( m \), whose density is given by $h_1(x) = \frac{1}{(1+x) \log 2}.$
	For further details, see \cite{Kva}, \cite{WP}, and \cite{WP2}. In particular, $f$ satisfies Corollaries \ref{oirtyuv} and \ref{coro1}, Proposition \ref{jkkgnhn} and Remark \ref{gfhjhfdf}.
\end{example}

\begin{example}[$\mathcal{P}$-Lüroth Maps]	
	\label{luroth}
	Let $\mathcal{P}=(I_i)_{i=1}^\infty$ be a countably infinite partition of $[0,1]$, consisting of non-empty, right-closed and left-open intervals. It
	is assumed throughout that the elements of $\mathcal{P}$ are ordered from right to left, starting from $I_1$, and that these elements accumulate only at $0$. Let $a_i := m(I_i)$ be the Lebesgue measure of $I_i$, and denote by $t_i := \sum_{k=i}^{\infty} a_k$ the Lebesgue measure of the $i$-th tail of $\mathcal{P}$. The $\mathcal{P}$-Lüroth map $f_\mathcal{P}:[0,1] \longrightarrow [0,1]$ is given by
	
	\begin{equation*}
		f_\mathcal{P} (x) = 
		\begin{cases}
			\dfrac{t_i-x}{a_i} , & \text{for all } x \in I_i \ \text{and all } i \geq 1 \\
			0, & \text{elsewhere}
		\end{cases}.
	\end{equation*}The $\mathcal{P}$-Lüroth map satisfies Definition \ref{def2} (see Remark \ref{remark1}).  In particular, $f$ satisfies Corollaries \ref{oirtyuv} and \ref{coro1}, and Proposition \ref{jkkgnhn}. For more details, see \cite{MSO}.
\end{example}

\subsubsection{Examples for the fiber map $G$.}
\begin{example}\label{h3}[Discontinuous Maps: constant coefficients] Let $F=(f, G)$ be the measurable map, where $f$ satisfies definitions \ref{def1} or \ref{def2} for some iterate $f^n$. Consider a sequence of real numbers $\{\alpha _i\}_{i=1}^\infty$ s.t $ 0\leq \alpha _i < \alpha _{i+1} \leq \alpha < 1$ for all $i$. Let  $G:[0,1] \times [0,1] \longrightarrow [0,1]$ be the function defined by $G(x,y)= \alpha_i y$ for all $x \in I_i$, for all $i \geq 1$ and all $y \in [0,1]$. It is straightforward to see that $G$ is discontinuous on the sets $\{\partial {I_i}\} \times K$ for all $i \geq 1$. Moreover, $G$ satisfies (H2) since $|G|_{\lip} =0$ (see equation (\ref{jdhfjdh})). Since $G$ is an $\alpha$-contraction, we have that $G$ also satisfies (H1). 
\end{example}

\begin{example}[Discontinuous Maps: Lipschitz coefficients]  Let $F=(f, G)$ be the measurable map, where $f$ satisfies definitions \ref{def1} or \ref{def2} for some iterate $f^n$ and denote the atoms of the partition $\mathcal{P}$ by $I_i:=(a_i,b_i)$ for all $i$ (as in definitions \ref{def1} and \ref{def2}). Consider a real number $0 \leq \alpha < 1$ and a sequence of real and Lipschitz functions $\{h_i\}_{i=1}^{\infty}$ such that $h_i:I_i \longrightarrow [0,1]$ for all $i \geq 1$, $h_i(b_i) \neq h_{i+1}(a_{i+1}) $, $ 0\leq h_i \leq \alpha < 1$ for all $i\geq 1$ and $\sup _{i\geq1} L(h_i) < \infty$ where $L(h_i)$ denotes the Lipschitz constant of $h_i$. Define $G:[0,1] \times [0,1] \to [0,1]$ by $G(x, y) = h_i(x) y$ for all $x \in I_i$ and $y \in [0,1]$.  It is easy to see that $G$ is discontinuous on the sets $\{\partial {I_i}\} \times K$ for all $i \geq 1$. Moreover, $G$ satisfies (H2) since $\sup _{i\geq1} L(h_i) < \infty$ (see equation (\ref{jdhfjdh})). Since, $G$ is an $\alpha$-contraction, we have that $G$ also satisfies (H1). 	
\end{example}

\begin{example}[(H2) is closed under iteration]
	Suppose that $F = (f, G)$, where $f$ is taken from Example~\ref{slopes2} and $G$ from Example~\ref{h3}. Then (H2) is closed under iteration, and (H3) is consequently satisfied. Moreover, the family $\{\alpha_i\}_i$ can be chosen so that the condition $\alpha \esssup g < 1$ is not satisfied at time one. In particular the second case of Remark \ref{uyrytuert} is satisfied, as well.
\end{example}

\subsection{Spectral gap for piecewise convex map with countable number of branches}

In this section, we present some results concerning the dynamics of $f$. Most of these results have already been discussed in \cite{RG2}, which is why their proofs are omitted from this text. For the results not covered in \cite{RG2} and whose proofs deviate from standard arguments, we will provide a proof.

\normalfont{A piecewise expanding map $f$ is non-singular with respect to $m$ and the Frobenious-Perron operator corresponding to $f$ is the linear operator $\func{P}_f: L^1_m \longrightarrow L^1_m$ that is given by the formula,
	\begin{equation}
	\func{P}_f h(x)=\sum_{i=1}^\infty \frac{h\left(f_i^{-1}(x)\right)}{|f'\left(f_i^{-1}(x)\right)|} \chi_{f(I_i)}(x)=\sum_{y\in\tau^{-1}(x)} h(y)g(y) \ \ \text{for} \ m-\text{a.e.} \ x \in I,
	\end{equation}for all $h \in L^1_m$.

	\begin{definition}\label{bv}
	Given $h:I \rightarrow \mathbb{R}$ we define variation of $h$ on a subset $J \subset I$ by
	\begin{equation*}
		V_J(h)=\sup\{\sum_{i=1}^k |h(x_i)-h(x_{i-1})|\},
	\end{equation*}
	where the supremum is taken over all finite sequences $(x_1,x_2,...x_k) \subset J$, where $x_1\leq x_2\leq...\leq x_k$. We need a variation $V(h)$ for $h\in L^1_m$, the set of all equivalence classes of real-valued, $m$-integrable functions on $I$.\\
	Let $BV_m =\{h\in L^1_m: V(h)<+\infty\}$, where $$V(h)=\inf\{ V_I ^*: \text{$h^*$ belongs to the equivalence class of $h$} \}.$$ We define for $h\in BV_m$,
	\begin{equation*}
		|h|_v= \int |h| dm + V(h).
	\end{equation*}
	\end{definition}
\begin{proposition} For every $h\in BV_m$ we have,
	\begin{equation}
		V_I \func{P}_{f^n} h \leq A_n V_I h +  B_n |h|_1,
	\end{equation}
	where $A_n=|g_n|_\infty+\max_{K\in \mathcal{Q}} V_K g_n < 1$, for $n$ sufficiently large, and $\displaystyle B_n=\frac{\max_{K\in \mathcal{Q}} V_K g_n}{m(K)}$ where
	
	\begin{equation*}
		g_n = 
		\begin{cases}
			\dfrac{1}{|(f^n)'|} , & \text{on } \displaystyle{\bigcup_{J\in \mathcal{P}^{(n)}} J} \\
			0, & \text{elsewhere}
		\end{cases}.
	\end{equation*}
\end{proposition}

\begin{proof}
	See \cite{RG2}.
\end{proof}

\begin{corollary}
If $f \in \mathcal{T}$ then for some $n>1$ and $h\in BV_m$, we have
\begin{equation*}
	|\func{P}_{f^n} h|_v \le r  |h|_v + C  |h|_1,
\end{equation*}
where $r\in(0,1)$ and $C>0$.
\end{corollary} Iterating the above inequality we arrive at the following result.
\begin{corollary}\label{oirtyuv1}
	If $f \in \mathcal{T}$ there exist constants $0\leq r_2<1$, $R_2\geq0$ and $C_2\geq0$ such that for all $n>1$ and all $h\in BV_m$, we have
	\begin{equation*}
		|\func{P}_{f}^n h|_v \le R_{2}r _{2}^{n}|h|_{v}+C_{2}|h|_{1}.
	\end{equation*}
\end{corollary}A measurable map \( f: I \to I \) is said to be \emph{non-singular} if the pushforward of the Lebesgue measure \( m \) by \( f \), denoted \( f_* m \), satisfies \( f_* m \ll m \).

The next corollary provides a generalization of the preceding result.

\begin{corollary}\label{oirtyuv}
	Suppose that $f:I \longrightarrow I$ is a non-singular map such that $\func{P}_{f}:BV_m \longrightarrow BV_m$ is bounded and $f^{n_0} \in \mathcal{T} \cup \mathcal{T}_E$ for some $n_0 \in \mathbb{N}$. Then, there exist constants $0\leq r_2<1$, $R_2\geq0$ and $C_2\geq0$ such that for all $n>1$ and all $h\in BV_m$, we have
	\begin{equation*}
		|\func{P}_{f}^n h|_v \le R_{2}r _{2}^{n}|h|_{v}+C_{2}|h|_{1}.
	\end{equation*}
\end{corollary} 

\begin{remark}\label{remark1}
	We note that in \cite{RG2}, the following Theorem \ref{teo1} was proven in a manner that applies to both classes of functions, $\mathcal{T}$ (Definition \ref{def1}) and $\mathcal{T}_E$ (Definition \ref{def2}). Moreover, this result remains valid even for systems that do not initially belong to $\mathcal{T}$ or $\mathcal{T}_E$ in the first iterate, but become members of one of these classes after some iteration (see Corollary \ref{coro1}). In other words, the following theorem holds for all systems such that $f^n \in  \mathcal{T} \cup \mathcal{T}_E$ for some $n$. See Examples \ref{gauss} and \ref{luroth}.
\end{remark}
%, provided there exists $h \in BV$ such that $\func{P}_f(h)=h$
% and has an invariant density in $BV$
\begin{theorem}\label{teo1}
For a piecewise convex map $f \in \mathcal{T}$ with countable number of branches, its Frobenius-Perron operator $\func{P}_f$ is quasi-compact on the space $BV_m$. More precisely, we have\\
(1) $\func{P}_f: L_m^1 \rightarrow L_m^1$ has $1$ as the only eigenvalue of modulus $1$.\\
(2) Set $E_1=\{h\in L_m^1 \mid \func{P}_f h=h\} \subseteq BV_m$ and $E_1$ is one-dimensional.\\
(3) $\displaystyle \func{P}_f= \Psi +Q$, where $\Psi$ represents the projection on eigenspace  $E_1$, $|\Psi|_1\leq 1$ and $Q$ is a linear operator on $L_m^1$ with $Q(BV_m) \subseteq BV_m$, $\displaystyle \sup_{n\in \mathbf{N}} |Q^n|_1 < \infty$ and $Q \cdot \Psi = \Psi \cdot Q  = 0$.\\
(4) $Q(BV_m)\subset BV_m$ and, considered as a linear operator on $(BV_m,|\cdot|_v)$, $Q$ satisfies $|Q^n|_v \leq H \cdot q^n$  $(n\geq 1)$ for some constants $H>0$ and $0<q<1$. 
\end{theorem}
\begin{proof}
	See \cite{RG2}.
\end{proof}

%and there exists a non-negative function $h_1\in BV_m$, s.t. $|h_1|_1=1$, which is a fixed point for the Frobenius-Perron operator $\func{P}_f$
\begin{corollary}\label{coro1}
Suppose that \( f: I \to I \) is a non-singular map such that the associated Frobenius–Perron operator $\mathcal{P}_f: BV_m \to BV_m$ is bounded and admits a fixed point \( m_1 \), which is a mixing probability measure with density \( h_1 \in BV_m \). If there exists \( n_0 \in \mathbb{N} \) such that \( f^{n_0} \in \mathcal{T} \cup \mathcal{T}_E \), then the operator \( \mathcal{P}_f \) is quasi-compact on the space \( BV_m \).

 More precisely, we have\\
	(1) $\func{P}_f: L_m^1 \rightarrow L_m^1$ has $1$ as the only eigenvalue of modulus $1$.\\
	(2) Set $E_1=\{h\in L_m^1 : \func{P}_f h=h\} \subseteq BV_m$ and $E_1$ is one dimensional.\\
	(3) $\displaystyle \func{P}_f= \Psi +Q$, where $\Psi$ represents the projection on eigenspace  $E_1$, $|\Psi|_1\leq 1$ and $Q$ is a linear operator on $L_m^1$ with $Q(BV_m) \subseteq BV_m$, $\displaystyle \sup_{n\in \mathbb{N}} |Q^n|_1 < \infty$ and $Q \cdot \Psi = \Psi \cdot Q  = 0$.\\
	(4) $Q(BV_m)\subset BV_m$ and, considered as a linear operator on $(BV_m,|\cdot|_v)$, $Q$ satisfies $|Q^n|_v \leq H \cdot q^n$  $(n\geq 1)$ for some constants $H>0$ and $0<q<1$. 
\end{corollary}

% and there exists a non-negative function $h_1\in BV_m$, s.t. $|h_1|_1=1$, which is a fixed point for the Frobenius-Perron operator $\func{P}_f$
\begin{proposition}\label{jkkgnhn}
Suppose that \( f: I \to I \) is a non-singular map such that the associated Frobenius–Perron operator $\mathcal{P}_f: BV_m \to BV_m$ is bounded and admits a fixed point \( m_1 \), which is a mixing probability measure with density \( h_1 \in BV_m \). If there exists \( n_0 \in \mathbb{N} \) such that \( f^{n_0} \in \mathcal{T} \cup \mathcal{T}_E \), then the Frobenius-Perron operator $\func{P}_f$ satisfies 
\begin{equation}\label{oriutrtrt}
	|\func{P}^n_f h|_v \leq H_2q^n |h|_v
\end{equation}for all $h \in BV_m$ such that $\int{h}dm=0$ and all $n \geq 1$.
\end{proposition}

\begin{proof}First, let us observe that for all $h \in BV_m$, it holds $\Psi(h)=h_1 \cdot  \int{h}dm$, where $h_1$ is a non-negative function $h_1\in BV_m$, s.t. $|h_1|_1=1$, which is a fixed point for the Frobenius-Perron operator $\func{P}_f$. Indeed, $\Psi$ is the projection on $E_1 = [h_1]$ (the space spanned by $h_1 \in BV_m$ such that $\func{P}_f(h_1)=h_1$). Thus $\Psi (h)= \lambda \cdot h_1$ for some $\lambda$ which depends on $h$. Then, we need to find an expression for $\lambda$. Integrating the relation $\displaystyle \func{P}_f^n (h)= \Psi (h) +Q^n(h)$ (which holds by (3), $\Psi \cdot Q = Q \cdot  \Psi = 0$) we get $$\displaystyle \func{P}_f^n (h)= \Psi (h) +Q^n(h).$$Hence, for all $n \geq 1$, it holds

\begin{eqnarray*}
	\lambda + \int{Q^n(h)} dm &=&  \int {\Psi(h) +Q^n(h)} dm \\&=&  \int {h} dm.
\end{eqnarray*} By (4), we get

\begin{eqnarray*}
\left| \int {Q^n(h)} dm \right| &\leq& H \cdot q^n, \ \text{for all } n\geq 1.
\end{eqnarray*}Taking the limit on both sides of the above relations, we get $\lambda = \int {h} dm$. Therefore, $\Psi (h) = h_1 \cdot \int h dm,$ for all $h \in BV_m$. To finish the proof, note that, it holds $\Psi(h)=0$ for every $h \in BV_m$ such that $\int{h}dm=0$, which proves Equation (\ref{oriutrtrt}).

\end{proof}

\begin{remark}\label{gfhjhfdf}
	Although the main motivation of this work lies in the class of piecewise convex maps \( f \in \mathcal{T} \), we emphasize that, with respect to the assumptions on \( f \), all the theorems in this article follow from the consequences of Corollary~\ref{coro1} and Proposition~\ref{jkkgnhn}. Therefore, the results established here apply to any transformation \( f: I \to I \) that is non-singular with respect to \( m \), for which the associated Frobenius-Perron operator \( \mathcal{P}_f: BV_m \to BV_m \) is bounded, admits a fixed point \( m_1 \), which is a mixing probability measure with density \( h_1 \in BV_m \) and there exists \( n_0 \in \mathbb{N} \) such that \( f^{n_0} \in \mathcal{T} \cup \mathcal{T}_E \).
\end{remark}

\section{The $\mathcal{L}^1$ and $S^1$ spaces and actions of $\func{F}_*$}\label{utrietyyighgdf}

In this section, we will construct the vector spaces that will be analyzed throughout the text; specifically, those on which we need to understand the action of the transfer operator associated with $F$. Since the foundation of this construction relies on the well-known Rohklin Disintegration Theorem, we begin by stating this result. From it, we will derive the key concepts underlying the vector spaces we will explore.

%\subsubsection{$L^{1}$-like space.}\label{spa}

\subsubsection*{Rokhlin's Disintegration Theorem}

Consider a probability space $(\Sigma,\mathcal{B}, \mu)$ and a partition $%
\Gamma$ of $\Sigma$ into measurable sets $\gamma \in \mathcal{B}$. Denote by $%
\pi : \Sigma \longrightarrow \Gamma$ the projection that associates to each
point $x \in \Sigma$ the element $\gamma _x$ of $\Gamma$ that contains $x$. That is, 
$\pi(x) = \gamma _x$. Let $\widehat{\mathcal{B}}$ be the $\sigma$-algebra of 
$\Gamma$ provided by $\pi$. Precisely, a subset $\mathcal{Q} \subset \Gamma$
is measurable if, and only if, $\pi^{-1}(\mathcal{Q}) \in \mathcal{B}$. We
define the \textit{quotient} measure $\mu _1$ on $\Gamma$ by $\mu _1(%
\mathcal{Q})= \mu(\pi ^{-1}(\mathcal{Q}))$.

The proof of the following theorem can be found in \cite{Kva}, Theorem
5.1.11 (items a), b) and c)) and Proposition 5.1.7 (item d)).

\begin{theorem}\label{rok}
	(Rokhlin's Disintegration Theorem) Suppose that $\Sigma $ is a complete and
	separable metric space, $\Gamma $ is a measurable partition of $\Sigma $ and $\mu $ is a probability on $\Sigma $. Then, $\mu $ admits a
	disintegration relative to $\Gamma $. That is, there exists a family $\{\mu _{\gamma}\}_{\gamma \in \Gamma }$ of probabilities on $\Sigma $ and a quotient measure $\mu _1:=\pi _*\mu$ (where $\pi$ is the canonical projection), such that:
	
	\begin{enumerate}
		\item[(a)] $\mu _\gamma (\gamma)=1$ for $\mu _1$-a.e. $\gamma \in \Gamma$;
		
		\item[(b)] for all measurable set $E\subset \Sigma $ the function $\Gamma
		\longrightarrow \mathbb{R}$ defined by $\gamma \longmapsto \mu _{\gamma
		}(E), $ is measurable;
		
		\item[(c)] for all measurable set $E\subset \Sigma $, it holds $\mu (E)=\int 
		{\mu _{\gamma }(E)}d\mu _1(\gamma )$.

		\item [(d)] If the $\sigma $-algebra $\mathcal{B}$ on $\Sigma $ has a countable
		generator, then the disintegration is unique in the following sense. If $(\{\mu _{\gamma }^{\prime }\}_{\gamma \in \Gamma },\mu _1)$ is another disintegration of the measure $\mu $ relative to $\Gamma $, then $\mu
		_{\gamma }=\mu _{\gamma }^{\prime }$, for $\mu _1$-almost every $\gamma
		\in \Gamma $.
	\end{enumerate}
\end{theorem}

\subsubsection{The $\mathcal{L}^{1}$ and $S^1$ spaces}\label{jdfjdhkjf}

Let $\mathcal{SB}(\Sigma )$ be the space of Borel signed measures on $\Sigma : = I \times K$. Given $\mu \in \mathcal{SB}(\Sigma )$, denote by $\mu ^{+}$ and $\mu ^{-}$
the positive and the negative parts of its Jordan decomposition, $\mu =\mu
^{+}-\mu ^{-}$ (see remark \ref{ghtyhh}). Let $\pi _{1}:\Sigma \longrightarrow I$ be the projection defined by $\pi_1 (x,y)=x$, denote by $\pi_{1*}: \mathcal{SB}(\Sigma) \rightarrow \mathcal{SB}(I)$ the pushforward map. Let $\mathcal{AB}$ denote the set of signed measures $\mu \in \mathcal{SB}(\Sigma )$ such that its associated positive and negative marginal measures, $\pi _{1*}\mu ^{+}$ and $\pi _{1*}\mu ^{-},$ are absolutely continuous with respect to $m$. That is,
	\begin{equation*}
		\mathcal{AB}=\{\mu \in \mathcal{SB}(\Sigma ):\pi _{1*}\mu ^{+}<<m\ \ 
		\mathnormal{and}\ \ \pi _{1* }\mu ^{-}<<m\}.  \label{thespace1}
	\end{equation*}%
}Given a \emph{probability measure} $\mu \in \mathcal{AB}$ on $\Sigma$,
Theorem \ref{rok} describes a disintegration $\left( \{\mu _{\gamma
}\}_{\gamma },\mu _{1}\right)$ along $\mathcal{F}^{s}$ by a family of probability measures $\{\mu _{\gamma }\}_{\gamma }$, defined on the stable leaves.  Moreover, since 
$\mu \in \mathcal{AB}$, $\mu _1$ can be identified with a non-negative
marginal density $\phi _{1}:I\longrightarrow \mathbb{R}$, defined almost
everywhere, where $|\phi _{1}|_{1}=1$. For a non-normalized
positive measure $\mu \in \mathcal{AB}$ we can define its disintegration following the same idea.  In this case, $\{ \mu _{\gamma } \}$ is still a family of probability measures, $\phi _{1}$ is still defined and $|\phi _{1}|_{1}=\mu (\Sigma)$.

\begin{definition}
	Let $\pi _{2}:\Sigma \longrightarrow K$ be the projection defined by $
	\pi _{2}(x,y)=y$. Consider $\pi
	_{\gamma ,2}:\gamma \longrightarrow K$, the restriction of the map $\pi
	_{2}$ to the vertical leaf $\gamma $, and the
	associated pushforward map $\pi _{\gamma ,2\ast }$. Given a positive measure 
	$\mu \in \mathcal{AB}$ and its disintegration along the stable leaves $%
	\mathcal{F}^{s}$, $\left( \{\mu _{\gamma }\}_{\gamma },\mu _{1}=\phi
	_{1}m\right) $, we define the \textbf{restriction of $\mu $ on $\gamma $}
	and denote it by $\mu |_{\gamma }$ as the positive measure on $K$ (not
	on the leaf $\gamma $) defined, for all measurable set $A\subset K$, as 
	\begin{equation*}
		\mu |_{\gamma }(A)=\pi _{\gamma ,2\ast }(\phi _{1}(\gamma )\mu _{\gamma
		})(A).
	\end{equation*}%
	For a given signed measure $\mu \in \mathcal{AB}$ and its Jordan
	decomposition $\mu =\mu ^{+}-\mu ^{-}$, define the \textbf{restriction of $%
		\mu $ on $\gamma $} by%
	\begin{equation*}
		\mu |_{\gamma }=\mu ^{+}|_{\gamma }-\mu ^{-}|_{\gamma }.
	\end{equation*}%
	\label{restrictionmeasure}
\end{definition}

\begin{remark}
	\label{ghtyhh}As proved in Appendix 2 of \cite {GLu},  restriction $%
	\mu |_{\gamma }$ does not depend on decomposition. Precisely, if $\mu
	=\nu _{1}-\nu _{2}$, where $\nu _{1}$ and $\nu _{2}$ are any positive
	measures, then $\mu |_{\gamma }=\nu _{1}|_{\gamma }-\nu _{2}|_{\gamma }$ $%
	\mu_{1}$-a.e. $\gamma \in I$. 
\end{remark}

Let $(X,d)$ be a compact metric space, $h:X\longrightarrow \mathbb{R}$ be a
Lipschitz function, and $L(h)$ be its best Lipschitz constant. That is,
\begin{equation}\label{lipsc}
	\displaystyle{L(h):=\sup_{x,y\in X,x\neq y}\left\{ \dfrac{|h(x)-h(y)|}{d(x,y)}%
		\right\} }.
\end{equation}

\begin{definition}
	Given two signed measures, $\mu $ and $\nu $ on $X,$ we define the \textbf{\
		Wasserstein-Kantorovich-like} distance between $\mu $ and $\nu $ by 
	\begin{equation*}
		W_{1}(\mu ,\nu ):=\sup_{L(h)\leq 1,|h|_{\infty }\leq 1}\left\vert \int {\
			h}d\mu -\int {h}d\nu \right\vert .
	\end{equation*}%
	\label{wasserstein}
\end{definition}We denote%

\begin{equation}
	||\mu ||_{W}:=W_{1}(0,\mu ),  \label{WW}
\end{equation}and observe that $||\cdot ||_{W}$ defines a norm on the vector space of signed measures defined on a compact metric space. It is worth remarking that this norm is equivalent to the standard norm of the dual space of Lipschitz functions.

\begin{definition}\label{sdfsdfsdasd}
	Let $\mathcal{L}^{1}\subseteq \mathcal{AB}(\Sigma )$ be the set of signed measures defined as%
	\begin{equation*}
		\mathcal{L}^{1}=\left\{ \mu \in \mathcal{AB}: \int W_{1}(\mu
		^{+}|_{\gamma },\mu ^{-}|_{\gamma }) dm <\infty \right\}.
	\end{equation*}%
	
	Define the function $||\cdot ||_{1}:\mathcal{L}^{1}\longrightarrow \mathbb{R}$ by%
	\begin{equation*}
		||\mu ||_{1}:= \int W_{1}(\mu
		^{+}|_{\gamma },\mu ^{-}|_{\gamma }) dm(\gamma) = \int ||\mu|_\gamma ||_W dm(\gamma).
	\end{equation*}Finally (see Definition \ref{bv}), consider the following set of signed measures on $\Sigma$%
	\begin{equation}\label{sinfi}
		S^{1}=\left\{ \mu \in \mathcal{L}^{1}:\phi _{1}\in
		BV_m \right\},
	\end{equation}%
	and the function, $||\cdot ||_{S^{1}}:S^{1}\longrightarrow 
	\mathbb{R}$, defined by (see Definition \ref{bv})%
	\begin{equation*}
		||\mu ||_{S^{1}}=|\phi _{1}|_{v}+||\mu ||_{1}.
	\end{equation*}
\end{definition}The sets $\left( \mathcal{L}^{1},||\cdot ||_{1}\right) $ and $\left(
S^{1},||\cdot||_{S^{1}}\right) $ are normed vector spaces. The proof of these facts are straightforward and a proof of an analogous case can be found in \cite{L}.

\subsection{The Transfer Operator Associated with $F$}
We begin by exploring properties of the action of $\func{F}_*$ on the spaces $S^1$ and $\mathcal{L}^1$, as defined in the previous section.

Consider the pushforward map (also known as the "transfer operator") $\func{F}_{\ast }$ associated with $F$, defined
by 
\begin{equation*}
	\lbrack \func{F}_{\ast }\mu ](E)=\mu (F^{-1}(E)),
\end{equation*}%
for each signed measure $\mu \in \mathcal{SB}(\Sigma )$ and for every measurable set $E \subset \Sigma$, where $\Sigma := I \times K$.

\begin{lemma}
	\label{transformula}	Let $F:\Sigma \longrightarrow \Sigma$ ($F=(f,G)$) be a transformation, where $f \in \mathcal{T}$ and $G$ satisfies (H1). Let $\mu \in \mathcal{AB}$ be a probability measure disintegrated by $(\{\mu _{\gamma }\}_{\gamma },\phi _1)$, the disintegration $(\{(\func{%
		F}_{\ast }\mu )_{\gamma }\}_{\gamma },(\func{F}_{\ast }\mu )_1)$ of the
	pushforward $\func{F}_{\ast }\mu $ \ satisfies the following relations%
	\begin{equation}
		(\func{F}_{\ast }\mu )_1=\func{P}_{f}(\phi _1)m  \label{1}
	\end{equation}
	and
	\begin{equation}
		(\func{F}_{\ast }\mu )_{\gamma }=\nu _{\gamma }:=\frac{1}{\func{P}_{f}(\phi
			_1)(\gamma )}\sum_{i=1}^{+\infty}{\frac{\phi _1}{|f'_{i}|}\circ
			f_{i}^{-1}(\gamma )\cdot \chi _{f_{i}(I_{i})}(\gamma )\cdot \func{F}_{\ast
			}\mu _{f_{i}^{-1}(\gamma )}}  \label{2}
	\end{equation}
	when $\func{P}_{f}(\phi _1)(\gamma )\neq 0$. Otherwise, if $\func{P}%
	_{f}(\phi _1)(\gamma )=0$, then $\nu _{\gamma }$ is the Lebesgue\footnote{Regarding the definition of \( \nu_{\gamma} \) on \( B := \{ \gamma \in I \mid \func{P}_f(\phi_1)(\gamma) = 0 \} \), the choice of the Lebesgue measure is not essential. Any other positive measure could be used instead without affecting the statement or the proof of Lemma \ref{transformula}.} measure on $\gamma $ (the expression $\displaystyle{\frac{\phi _1}{|f'_{i}|}%
		\circ f_{i}^{-1}(\gamma )\cdot \frac{\chi _{f_{i}(I_{i})}(\gamma )}{\func{P}%
			_{f}(\phi _1)(\gamma )}\cdot \func{F}_{\ast }\mu _{f_{i}^{-1}(\gamma )}}$
	is understood to be zero outside $f_{i}(I_{i})$ for all $i=1, 2, \cdots$).
	Here and above, $\chi _{A}$ is the characteristic function of the set $A$.
\end{lemma}

	\begin{proof}
	By the uniqueness of the disintegration (see item (d) of Theorem~\ref{rok}), it is enough to prove the following equation 
		\begin{equation*}
			\func{F}_{\ast }\mu (E)=\int_{I}{\nu _{\gamma }(E\cap \gamma )}\func{P}%
			_{f}(\phi _1)(\gamma )dm(\gamma) ,
		\end{equation*}%
		for a measurable set $E\subset \Sigma $. For this purpose, let us define the
		sets $B_{1}=\left\{ \gamma \in I:f^{-1}(\gamma )=\emptyset \right\} $, $%
		B_{2}=\left\{ \gamma \in B_{1}^{c}:\func{P}_{f}(\phi _1)(\gamma
		)=0\right\} $ and $B_{3}=\left( B_{1}\cup B_{2}\right) ^{c}$. The following
		properties can be easily proven:
		
		\begin{enumerate}
			\item[(1)] $B_i \cap B_j = \emptyset$, $f^{-1}(B_i) \cap f^{-1}(B_j) =
			\emptyset$, for all $1\leq i,j \leq 3$ such that $i \neq j$ and $\bigcup
			_{i=1} ^{3} {B_i} = \bigcup _{i=1} ^{3} {f^{-1}(B_i)} =I$;
			
			\item[(2)] $m_{1}(f^{-1}(B_{1}))=\phi _1m(f^{-1}(B_{2}))=0$;
		\end{enumerate}
		
		Using the change of variables $\gamma =f_{i}(\beta )$ and the definition of $%
		\nu _{\gamma }$ (see (\ref{2})), we have 
		\begin{equation*}
			\begin{split}
				&\int_{I}{\nu _{\gamma }(E\cap \gamma )}\func{P}_{f}(\phi _1)(\gamma
				)dm(\gamma) \\
				=&\int_{B_{3}}{\sum_{i=1}^{\infty}{\ {\frac{\phi _1}{|f'_{i}|}%
							\circ f_{i}^{-1}(\gamma )\func{F}_{\ast }\mu _{f_{i}^{-1}(\gamma )}(E)\chi
							_{f_{i}(I_{i})(\gamma )}}}}dm(\gamma ) \\
				=&\sum_{i=1}^{\infty}{\int_{f_{i}(I_{i})\cap B_{3}}{\ {\frac{\phi _1}{|f'_{i}|}\circ f_{i}^{-1}(\gamma )\func{F}_{\ast }\mu _{f_{i}^{-1}(\gamma
								)}(E)}}}dm(\gamma ) \\
				=&\sum_{i=1}^{\infty}{\int_{I_{i}\cap f_{i}^{-1}(B_{3})}{\ {\phi _1(\beta )\mu
							_{\beta }(F^{-1}(E))}}}dm(\beta ) \\
				=&{\int_{f^{-1}(B_{3})}{\ {\phi _1(\beta )\mu _{\beta }(F^{-1}(E))}}}%
				dm(\beta ) \\
				=&\int_{\bigcup_{i=1}^{3}{f^{-1}(B_{i})}}{\ {\ \mu _{\beta }(F^{-1}(E))}}%
				d\phi _1m(\beta ) \\
				=&\int_{I}{\ {\ \mu _{\beta }(F^{-1}(E))}}d\phi _1m(\beta ) \\
				=&\mu (F^{-1}(E)) \\
				=&\func{F}_{\ast }\mu (E).
			\end{split}
		\end{equation*}%
	This completes the proof.
	\end{proof}
	
	The next proposition follows directly from Definition~\ref{restrictionmeasure} and Lemma~\ref{transformula}, and we omit its proof.

\begin{proposition}
	\label{niceformulaab}	Let $F:\Sigma \longrightarrow \Sigma$ ($F=(f,G)$) be a transformation, where $f \in \mathcal{T}$ and $G$ satisfies (H1). Let $\gamma \in \mathcal{F}^{s}$ be a stable leaf. Define the map $F_{\gamma }:K\longrightarrow K$ by 
	\begin{equation}\label{ritiruwt}
		F_{\gamma }=\pi _{2}\circ F|_{\gamma }\circ \pi _{\gamma ,2}^{-1}.
	\end{equation}%
	Then, for each $\mu \in \mathcal{L}^{1}$ and for almost all $\gamma \in
	I$ it holds 
	\begin{equation}
		(\func{F}_{\ast }\mu )|_{\gamma }=\sum_{i=1}^{+\infty}{\func{F}%
			_{\gamma _i \ast }\mu |_{\gamma _i }g _i(\gamma _i)\chi _{f_{i}(I_{i})}(\gamma )}\ \ m%
		\mathnormal{-a.e.}\ \ \gamma \in I  \label{niceformulaa}
	\end{equation}%
	where $\func{F}_{\gamma_i \ast }$ is the pushforward map
	associated to $\func{F}_{\gamma_i}$, $\gamma _i = f_{i}^{-1}(\gamma )$ when $\gamma \in f_i (I_i)$ and $g_i(\gamma)= \dfrac{1}{|f_i^{'}(\gamma)|}$, where $f_i = f|_{I_i}$.
\end{proposition}

\subsection{Basic Properties and Convergence to the equilibrium}\label{invt}

\begin{lemma}
	\label{niceformulaac}	Let $F:\Sigma \longrightarrow \Sigma$ ($F=(f,G)$) be a transformation, where $f \in \mathcal{T}$ and $G$ satisfies (H1). For every $\mu \in \mathcal{AB}$ and a stable leaf $%
	\gamma \in \mathcal{F}^{s}$, it holds 
	\begin{equation}
		||\func{F}_{\gamma \ast }\mu |_{\gamma }||_{W}\leq ||\mu |_{\gamma }||_{W},
		\label{weak1}
	\end{equation}%
	where $F_{\gamma }:K\longrightarrow K$ is defined in Proposition \ref%
	{niceformulaab} and $\func{F}_{\gamma \ast }$ is the associated pushforward
	map. Moreover, if $\mu$ is a probability measure on $K$, it holds 
	\begin{equation}
		||\func{F}_{\gamma \ast }{^{n}}\mu ||_{W}=||\mu ||_{W}=1,\ \ \forall \ \
		n\geq 1.  \label{simples}
	\end{equation}
\end{lemma}

\begin{proof}
	Indeed, since $F_{\gamma }$ is an $\alpha $%
	-contraction, if $|h|_{\infty }\leq 1$ and $L(h)\leq 1$ the same holds for 
	$h\circ F_{\gamma }$. Since 
	\begin{equation*}
		\left\vert \int {h~}d\func{F}_{\gamma \ast }\mu |_{\gamma }\right\vert
		=\left\vert \int {h(F_{\gamma })~}d\mu |_{\gamma }\right\vert ,
	\end{equation*}%
	taking the supremum over $h$ such that $|h|_{\infty }\leq 1$ and $%
	L(h)\leq 1$ we finish the proof.
	
	In order to prove equation (\ref{simples}), consider a probability measure $%
	\mu $ on $K$ and a Lipschitz function $h:K\longrightarrow \mathbb{R}$%
	, such that $||h||_{\infty }\leq 1$ we get immediately that $|\int {h}d\mu |\leq
	||h||_{\infty }\leq 1$, which yields $||\mu ||_{W}\leq 1$. Considering $%
	h\equiv 1$ we get $||\mu ||_{W}=1$.
\end{proof}

\begin{proposition}[Weak contraction on $\mathcal{L}^1$]Let $F:\Sigma \longrightarrow \Sigma$ ($F=(f,G)$) be a transformation, where $f \in \mathcal{T}$ and $G$ satisfies (H1). If $\mu \in \mathcal{L}^{1}$, then $||\func{F}_{\ast }\mu ||_{1}\leq ||\mu ||_{1}$.
	\label{weakcontral11234}
\end{proposition}
\begin{proof}
	In the following, we consider for all $i$, the change of variable $\gamma
	=f_{i}(\beta)$. Thus, Lemma \ref{niceformulaac} and equation (\ref%
	{niceformulaa}) yield 
	\begin{eqnarray*}
		||\func{F}_{\ast }\mu ||_{1} &=&\int_{I}{\ ||(\func{F}_{\ast }\mu
			)|_{\gamma }||_{W}}dm(\gamma ) \\
		&\leq &\sum_{i=1}^{+\infty}{\int_{f(I_{i})}{\ \left\vert \left\vert \dfrac{\func{F}%
					_{f_{i}^{-1}(\gamma )\ast }\mu |_{f_{i}^{-1}(\gamma )}}{|f'_{i}(f_{i}^{-1}(\gamma ))|}\right\vert \right\vert _{W}}dm(\gamma )} \\
		&=&\sum_{i=1}^{+\infty}{\int_{I_{i}}{\left\vert \left\vert \func{F}_{\beta \ast
				}\mu |_{\beta }\right\vert \right\vert _{W}}dm(\beta )} \\
		&=&\sum_{i=1}^{+\infty}{\int_{I_{i}}{\left\vert \left\vert \mu |_{\beta
				}\right\vert \right\vert _{W}}}dm(\beta) \\
		&=&||\mu ||_{1}.
	\end{eqnarray*}
\end{proof}
\begin{proposition}[Lasota-Yorke inequality]
	Let $F:\Sigma \longrightarrow \Sigma$ ($F=(f,G)$) be a transformation, where $f \in \mathcal{T}$ and $G$ satisfies (H1). Then, for all $\mu \in S^{1}$, it holds%
	\begin{equation}
		||\func{F}_{\ast }^{n}\mu ||_{S^{1}}\leq R_2r_2 ^{n}||\mu
		||_{S^{1}}+(C_2 +1)||\mu ||_{1},\ \ \forall n\geq 1,  \label{xx}
	\end{equation}where the constants $R_2, r_2$ and $C_2$ are from Corollary \ref{oirtyuv}.
	\label{lasotaoscilation2}
\end{proposition}

\begin{proof}
	
	Let $\phi _1$ be the marginal density of the
	disintegration of $\mu$. Precisely, $\phi _1=\phi _1^{+}-\phi _1^{-}$%
	, where $\phi _1^{+}=\dfrac{d\pi _1^{\ast }\mu ^{+}}{dm}$ and $\phi
	_1^{-}=\dfrac{d\pi _1^{\ast }\mu ^{-}}{dm}$. By the definition of
	the Wasserstein norm it follows that for every $\gamma $ it holds $||\mu
	|_{\gamma }||_{W}\geq \int 1~d(\mu |_{\gamma })=\phi _1(\gamma )$. Thus, 
	$|\phi_1|_{1}\leq ||\mu ||_{1}.$ This estimate, together with Corollary \ref{oirtyuv} and Proposition \ref{weakcontral11234}, implies that 
	\begin{eqnarray*}
		||\func{F}_{\ast }^{n}\mu ||_{S^{1}} &=&|\func{P}_{f}^{n}\phi _1|_{v}+||%
		\func{F}_{\ast }^{n}\mu ||_{1} \\
		&\leq &R_2r _{2}^{n}|\phi_1|_{v}+C_{2}|\phi_1|_{1}+||\mu ||_{1}
		\\
		&\leq &R_2r_{2}^{n}||\mu ||_{S^{1}}+(C_{2}+1)||\mu ||_{1} \ \ \forall \ n \geq 1.
	\end{eqnarray*}%
\end{proof}

\subsubsection{Convergence to equilibrium}

\begin{lemma}
	Let $F:\Sigma \longrightarrow \Sigma$ ($F=(f,G)$) be a transformation, where $f \in \mathcal{T}$ and $G$ satisfies (H1). For all signed measures $\mu $ on $K$ and for all $\gamma \in \mathcal{F}^s$, it holds%
	\begin{equation*}
		||\func{F}_{\gamma \ast }\mu ||_{W}\leq \alpha ||\mu ||_{W}+|\mu (K)|
	\end{equation*}%
	($\alpha $ is the rate of contraction of $G$, see (H1)). In
	particular, if $\mu (K)=0$ then%
	\begin{equation*}
		||\func{F}_{\gamma \ast }\mu ||_{W}\leq \alpha ||\mu ||_{W},
	\end{equation*}which provides a contraction on the space of zero-mass measures.
	\label{quasicontract}
\end{lemma}

\begin{proof}
	If $L(h)\leq 1$ and $||h||_{\infty }\leq 1$, then $h\circ F_{\gamma }$ is $%
	\alpha $-Lipschitz. Moreover, since $||h||_{\infty }\leq 1$, then $||h\circ
	F_{\gamma }-\theta ||_{\infty }\leq \alpha $, for some $\theta $ such that $%
	|\theta |\leq 1$. Indeed, let $z\in K$ be such that $|h\circ F_{\gamma
	}(z)|\leq 1$, set $\theta =h\circ F_{\gamma }(z)$ and let $d_{2}$ be the
	 metric on $K$. Since $\diam(K)=1$, we have 
	\begin{equation*}
		|h\circ F_{\gamma }(y)-\theta |\leq \alpha d_{2}(y,z)\leq \alpha
	\end{equation*}%
	and consequently $||h\circ F_{\gamma }-\theta ||_{\infty }\leq \alpha $.
	
	This implies,
	
	\begin{align*}
		\left\vert \int_{K}{h}d\func{F}_{\gamma \ast }\mu \right\vert &
		=\left\vert \int_{K}{h\circ F_{\gamma }}d\mu \right\vert \\
		& \leq \left\vert \int_{K}{h\circ F_{\gamma }-\theta }d\mu \right\vert
		+\left\vert \int_{K}{\theta }d\mu \right\vert \\
		& =\alpha \left\vert \int_{K}{\frac{h\circ F_{\gamma }-\theta }{\alpha }}%
		d\mu \right\vert +|\theta ||\mu (K)|.
	\end{align*}%
	And taking the supremum over $h$ such that $|h|_{\infty }\leq 1$ and $%
	L(h)\leq 1$ we have $||\func{F}_{\gamma \ast }\mu ||_{W}\leq \alpha ||\mu
	||_{W}+\mu (K)$. In particular, if $\mu (K)=0$, we get the second
	part.
\end{proof}

\begin{proposition}
	\label{5.6} 	Let $F:\Sigma \longrightarrow \Sigma$ ($F=(f,G)$) be a transformation, where $f \in \mathcal{T}$ and $G$ satisfies (H1). For all signed measures $\mu \in \mathcal{L}^{1}$, it holds 
	\begin{equation}
		||\func{F}_{\ast }\mu ||_{1}\leq \alpha ||\mu ||_{1}+(\alpha +1)|\phi
		_1|_{1}.  \label{abovv}
	\end{equation}
\end{proposition}

\begin{proof}
	Consider a signed measure $\mu \in \mathcal{L}^{1}$ and its restriction on
	the leaf $\gamma $, $\mu |_{\gamma }=\pi _{2,\gamma\ast }(\phi _1(\gamma
	)\mu _{\gamma })$. Set%
	\begin{equation*}
		\overline{\mu }|_{\gamma }=\pi _{2,\gamma \ast }\mu _{\gamma }.
	\end{equation*}%
	If $\mu $ is a positive measure then $\overline{\mu }|_{\gamma }$ is a
	probability on $K$ and $\mu |_{\gamma }=\phi _1(\gamma )\overline{\mu }%
	|_{\gamma }$. Then, the expression given by Proposition \ref{niceformulaab}
	yields%
	\begin{equation*}
		\begin{split}
			&||\func{F}_{\ast }\mu ||_{1} \\
			&\leq \sum_{i=1}^{\infty}{\ \int_{f(I_{i})}{\
					\left\vert \left\vert \frac{\func{F}_{f_{i}^{-1}(\gamma )\ast }\overline{\mu
							^{+}}|_{f_{i}^{-1}(\gamma )}\phi _1^{+}(f_{i}^{-1}(\gamma ))}{|f'_{i}|\circ f_{i}^{-1}(\gamma )}-\frac{\func{F}_{f_{i}^{-1}(\gamma )\ast }%
						\overline{\mu ^{-}}|_{f_{i}^{-1}(\gamma )}\phi _1^{-}(f_{i}^{-1}(\gamma ))%
					}{|f'_{i}|\circ f_{i}^{-1}(\gamma )}\right\vert \right\vert _{W}}%
				dm(\gamma )} \\
			&\leq \func{I}_{1}+\func{I}_{2},
		\end{split}
	\end{equation*}%
	where%
	\begin{equation*}
		\func{I}_{1}=\sum_{i=1}^{\infty}{\ \int_{f(I_{i})}{\ \left\vert \left\vert \frac{%
					\func{F}_{f_{i}^{-1}(\gamma )\ast }\overline{\mu ^{+}}|_{f_{i}^{-1}(\gamma
						)}\phi _1^{+}(f_{i}^{-1}(\gamma ))}{|f'_{i}|\circ f_{i}^{-1}(\gamma )}%
				-\frac{\func{F}_{f_{i}^{-1}(\gamma )\ast }\overline{\mu ^{+}}%
					|_{f_{i}^{-1}(\gamma )}\phi _1^{-}(f_{i}^{-1}(\gamma ))}{|f'_{i}|\circ f_{i}^{-1}(\gamma )}\right\vert \right\vert _{W}}dm(\gamma )%
		}
	\end{equation*}%
	and%
	\begin{equation*}
		\func{I}_{2}=\sum_{i=1}^{\infty}{\ \int_{f(I_{i})}{\ \left\vert \left\vert \frac{%
					\func{F}_{f_{i}^{-1}(\gamma )\ast }\overline{\mu ^{+}}|_{f_{i}^{-1}(\gamma
						)}\phi _1^{-}(f_{i}^{-1}(\gamma ))}{|f'_{i}|\circ f_{i}^{-1}(\gamma )}%
				-\frac{\func{F}_{f_{i}^{-1}(\gamma )\ast }\overline{\mu ^{-}}%
					|_{f_{i}^{-1}(\gamma )}\phi _1^{-}(f_{i}^{-1}(\gamma ))}{|f'_{i}|\circ f_{i}^{-1}(\gamma )}\right\vert \right\vert _{W}}dm(\gamma )%
		}.
	\end{equation*}%
	In the following we estimate $\func{I}_{1}$ and $\func{I}_{2}$. By Lemma \ref%
	{niceformulaac} and a change of variable we have 
	\begin{eqnarray*}
		\func{I}_{1} &=&\sum_{i=1}^{\infty}{\ \int_{f(I_{i})}{\ \left\vert \left\vert 
				\func{F}_{f_{i}^{-1}(\gamma )\ast }\overline{\mu ^{+}}|_{f_{i}^{-1}(\gamma
					)}\right\vert \right\vert _{W}\frac{|\phi _1^{+}-\phi _1^{-}|}{|f'_{i}|}\circ f_{i}^{-1}(\gamma )}dm(\gamma )} \\
		&\leq &\int_{I}{\ \left\vert \left\vert \func{F}_{\beta \ast }\overline{%
				\mu ^{+}}|_{\beta }\right\vert \right\vert _{W}|\phi _1^{+}-\phi
			_1^{-}|(\beta )}dm(\beta ) \\
		&=&\int_{I}{\ |\phi _1^{+}-\phi _1^{-}|(\beta )}dm(\beta ) \\
		&=&|\phi_1|_{1},
	\end{eqnarray*}%
	and by Lemma \ref{quasicontract} we have%
	\begin{eqnarray*}
		\func{I}_{2} &=&\sum_{i=1}^{\infty}{\ \int_{f(I_{i})}{\ \left\vert \left\vert 
				\func{F}_{f_{i}^{-1}(\gamma )\ast }\left( \overline{\mu ^{+}}%
				|_{f_{i}^{-1}(\gamma )}-\overline{\mu ^{-}}|_{f_{i}^{-1}(\gamma )}\right)
				\right\vert \right\vert _{W}\frac{\phi _1^{-}}{|f'_{i}|}\circ
				f_{i}^{-1}(\gamma )}dm(\gamma )} \\
		&\leq &\sum_{i=1}^{\infty}{\ \int_{I_{i}}{\ \left\vert \left\vert \func{F}_{\beta
					\ast }\left( \overline{\mu ^{+}}|_{\beta }-\overline{\mu ^{-}}|_{\beta
				}\right) \right\vert \right\vert _{W}\phi _1^{-}(\beta )}dm(\beta )} \\
		&\leq &\alpha \int_{I}{\ \left\vert \left\vert \overline{\mu ^{+}}%
			|_{\beta }-\overline{\mu ^{-}}|_{\beta }\right\vert \right\vert _{W}\phi
			_1^{-}(\beta )}dm(\beta ) \\
		&\leq &\alpha \int_{I}{\ \left\vert \left\vert \overline{\mu ^{+}}%
			|_{\beta }\phi _1^{-}(\beta )-\overline{\mu ^{-}}|_{\beta }\phi
			_1^{-}(\beta )\right\vert \right\vert _{W}}dm(\beta ) \\
		&\leq &\alpha \int_{I}{\ \left\vert \left\vert \overline{\mu ^{+}}%
			|_{\beta }\phi _1^{-}(\beta )-\overline{\mu ^{+}}|_{\beta }\phi
			_1^{+}(\beta )\right\vert \right\vert _{W}}dm(\beta )\\&+&\alpha
		\int_{I}{\ \left\vert \left\vert \overline{\mu ^{+}}|_{\beta }\phi
			_1^{+}(\beta )-\overline{\mu ^{-}}|_{\beta }\phi _1^{-}(\beta
			)\right\vert \right\vert _{W}}dm(\beta ) \\
		&=&\alpha |\phi _1|_{1}+\alpha ||\mu ||_{1}.
	\end{eqnarray*}%
	Summing the above estimates we finish the proof.
\end{proof}

Iterating (\ref{abovv}) we get the following corollary.

\begin{corollary}
	Let $F:\Sigma \longrightarrow \Sigma$ ($F=(f,G)$) be a transformation, where $f \in \mathcal{T}$ and $G$ satisfies (H1). For all signed measures $\mu \in \mathcal{L}^{1}$ it holds 
	\begin{equation*}
		||\func{F}_{\ast }^{n}\mu ||_{1}\leq \alpha ^{n}||\mu ||_{1}+\overline{%
			\alpha }|\phi _1|_{1},
	\end{equation*}%
	where $\overline{\alpha }=\frac{1+\alpha }{1-\alpha }$. \label{nicecoro}
\end{corollary}

Let us consider the set of zero average measures in $S^{1}$ defined by 
\begin{equation}
	\mathcal{V}_{s}=\{\mu \in S^{1}:\mu (\Sigma )=0\}.  \label{mathV}
\end{equation}%
Note that, for all $\mu \in \mathcal{V}_{s}$ we have $\pi _{1*}\mu
(I)=0$. Moreover, since $\pi _{1*}\mu =\phi _1m$ ($\phi
_1=\phi _1^{+}-\phi _1^{-}$), we have $\displaystyle{\int_{I}{\phi
		_1}dm=0}$. 

\begin{proposition}[Exponential convergence to equilibrium]
	\label{5.8} 	Let $F:\Sigma \longrightarrow \Sigma$ ($F=(f,G)$) be a transformation, where $f \in \mathcal{T}$ and $G$ satisfies (H1). There exist $D_{2}\in \mathbb{R}$ and $0<\beta _{1}<1$ such that
	for every signed measure $\mu \in \mathcal{V}_{s}$, it holds 
	\begin{equation*}
		||\func{F}_{\ast }^{n}\mu ||_{1}\leq D_{2}\beta _{1}^{n}||\mu ||_{S^{1}},
	\end{equation*}%
	for all $n\geq 1$. \label{quasiquasiquasi}
\end{proposition}

\begin{proof}
	Given $\mu \in \mathcal{V}_s$ and denoting $\phi _1=\phi _1^{+}-\phi
	_1^{-}$, it holds that $\int {\phi }_1dm=0$. Moreover, Proposition \ref%
	{jkkgnhn} yields $|\func{P}_{f}^{n}(\phi _1)|_{v}\leq H_2q^{n}|\phi _1|_{v}$
	for all $n\geq 1$, then $|\func{P}_{f}^{n}(\phi _1)|_{1}\leq H_2q^{n}||\mu
	||_{S^{1}}$ for all $n\geq 1$.
	
	Let $l$ and $0\leq d\leq 1$ be the coefficients of the division of $n$ by $2$%
	, i.e. $n=2l+d$. Thus, $l=\frac{n-d}{2}$ (by Proposition \ref%
	{weakcontral11234}, we have $||\func{F}_{\ast }^{n}\mu ||_{1}\leq ||\mu
	||_{1}$, for all $n$, and $||\mu ||_{1}\leq ||\mu ||_{S^{1}}$) and by
	Corollary \ref{nicecoro}, it holds (below, set $\beta _{1}=\max \{\sqrt{q},%
	\sqrt{\alpha }\}$)
	
	\begin{eqnarray*}
		||\func{F}_{\ast }^{n}\mu ||_{1} &\leq &||\func{F}_{\ast }^{2l+d}\mu ||_{1}
		\\
		&\leq &\alpha ^{l}||\func{F}_{\ast }^{l+d}\mu ||_{1}+\overline{\alpha }%
		\left\vert \dfrac{d(\pi _{1*}(\func{F}^{\ast l+d}\mu ))}{dm}%
		\right\vert _{1} \\
		&\leq &\alpha ^{l}||\mu ||_{1}+\overline{\alpha }|\func{P}_{f}^{l}(\phi
		_1)|_{1} \\
		&\leq &\alpha ^{l}||\mu ||_{1}+\overline{\alpha }H_2 q^l||\mu||_ {S^1} \\
		&= &\alpha ^{\frac{n-d}{2}}||\mu ||_{1}+\overline{\alpha }H_2 q^{\frac{n-d}{2}}||\mu||_ {S^1}  \\ 
		&\leq &\alpha ^{\frac{-d}{2}} \alpha ^{\frac{n}{2}}||\mu ||_{S^1}+\overline{\alpha }H_2 q^{\frac{-d}{2}} q^ {\frac{n}{2}}||\mu||_ {S^1}  \\
		&\leq &(1+\overline{\alpha }H_2)\beta _{1}^{-d}\beta _{1}^{n}||\mu ||_{S^{1}}
		\\
		&\leq &D_{2}\beta _{1}^{n}||\mu ||_{S^{1}},
	\end{eqnarray*}%
	where $D_{2}=\dfrac{1+\overline{\alpha }H_2}{\beta _{1}}$.
\end{proof}

Now we show that under the assumptions taken, the system has a unique
invariant measure $\mu _{0}\in S^{1}$.

\begin{athm}\label{gfhduer}	Let $F:\Sigma \longrightarrow \Sigma$, $F=(f,G)$, be a transformation, where $f \in \mathcal{T}$ and $G$ satisfies (H1). Then, it has a unique invariant probability in $S^{1}$.
\end{athm}
\begin{proof}
	Let $\mu _{0}$ be the $F$-invariant measure 
	such that $\pi _{\ast }\mu _{0}=m_1$ which do exist by Proposition \ref{kjdhkskjfkjskdjf}. Suppose that $%
	h:K\longrightarrow \mathbb{R}$ is a Lipschitz function such that $%
	|h|_{\infty }\leq 1$ and $L(h)\leq 1$. Then, it holds $\left\vert \int {h}%
	d(\mu _{0}|_{\gamma })\right\vert \leq |h|_{\infty }\leq
	1$. Hence, $\mu _{0}\in \mathcal{L}^{1}$. Since, $\dfrac{\pi_{1*}\mu_0}{dm} \equiv h_1 \in BV_m$, we have $\mu_0 \in S^1$.
	
	The uniqueness follows directly from Proposition \ref{5.8}, since the difference between two probabilities ($\mu _1 - \mu_0$) is a zero average signed measure and both are fixed points of $\func{F}_ *$.
\end{proof}

\subsection{Spectral Gap}\label{jshdjfgsjhdf}

\begin{athm}[Spectral Gap for $\func{F}_*$ on $S^1$]
	\label{spgap}Let $F:\Sigma \longrightarrow \Sigma$ ($F=(f,G)$) be a transformation, where $f \in \mathcal{T}$ and $G$ satisfies (H1). Then, the operator $\func{F}_{\ast
	}:S^{1}\longrightarrow S^{1}$ can be written as 
	\begin{equation*}
		\func{F}_{\ast }=\func{P}+\func{N},
	\end{equation*}%
	where
	
	\begin{enumerate}
		\item[a)] $\func{P}$ is a projection, i.e., $\func{P}^2 = \func{P}$ and $\dim \operatorname{Im}(\func{P}) = 1$;

		\item[b)] there are $0\leq \lambda _0 <1$ and $U\geq0$ such that $\forall \mu \in S^1$ 
		\begin{equation*}
			||\func{N}^{n}(\mu )||_{S^{1}}\leq ||\mu||_{S^{1}} \lambda _0 ^{n}U;
		\end{equation*}
		
		\item[c)] $\func{P}\func{N}=\func{N}\func{P}=0$.
	\end{enumerate}
\end{athm}

\begin{proof}
	First, let us show there exist $0<\lambda_0 <1$ and $U_{1}>0$ such that, for all $%
	n\geq 1$, it holds 
	\begin{equation*}
		||\func{F}_{\ast }^{n}||_{{\mathcal{V}_{s}}\rightarrow {\mathcal{V}_{s}}%
		}\leq \lambda _0 ^{n}U_{1}  
	\end{equation*}%
	where ${\mathcal{V}_{s}}$ is the zero average space defined in $($\ref{mathV}%
	$)$ and $\lambda _{0}:=\max \{\sqrt{\beta
		_{1}},\sqrt{r_2} \}$. Indeed, consider $\mu \in \mathcal{V}_{s}$ (see \eqref{mathV}) s.t. $%
	||\mu ||_{S^{1}}\leq 1$ and for a given $n\in \mathbb{N}$ let $m$ and $0\leq
	d\leq 1$ be the coefficients of the division of $n$ by $2$, i.e. $n=2m+d$.
	Thus $m=\frac{n-d}{2}$. By the Lasota-Yorke inequality (Proposition \ref%
	{lasotaoscilation2}) we have the uniform bound $||\func{F}_{\ast }^{n}\mu
	||_{S^{1}}\leq R_{2}+C_2+1$ for all $n\geq 1$. Moreover, by Propositions \ref%
	{quasiquasiquasi} and \ref{weakcontral11234} there is some $D_{2}$ such that
	it holds (below, let $\lambda _{0}$ be defined by $\lambda _{0}=\max \{\sqrt{\beta
		_{1}},\sqrt{r_2} \}$)%
	\begin{eqnarray*}
		||\func{F}_{\ast }^{n}\mu ||_{S^{1}} &\leq &R_2r_2 ^{m}||\func{F}_{\ast
		}^{m+d}\mu ||_{S^{1}}+(C_2 +1)||\func{F}_{\ast }^{m+d}\mu ||_{1} \\
		&\leq &R_2r_2 ^{m} (R_2 + C_2 +1) +(C_2+1) ||\func{F}_{\ast }^{m}\mu ||_{1} 
		\\&\leq &r_2 ^{m} R_2(R_2 + C_2 +1) +(C_2+1)D_2\beta _1^m
		\\ &\leq &r_2 ^{m} R_2(R_2 + C_2 +1) +(C_2+1)D_2\beta _1^m
		\\ &\leq &r_2 ^{\frac{n-d}{2}} R_2(R_2 + C_2 +1) +(C_2+1)D_2\beta _1^{\frac{n-d}{2}}
		\\ &\leq &r_ 2 ^{\frac{-d}{2}} r_2 ^{\frac{n}{2}} R_2(R_2 + C_2 +1) +(C_2+1)D_2 \beta_1 ^{\frac{-d}{2}}\beta _1^{\frac{n}{2}}
		\\&\leq &\lambda _{0}^{n}\left[  r_ 2 ^{-\frac{1}{2}}R_2(R_2 + C_2 +1) +  (C_2+1)D_2 \beta_1 ^{-\frac{1}{2}}  \right] 
		\\&\leq &\lambda _{0}^{n}U_1,
	\end{eqnarray*}%
	where $U_{1}=\left[  r_ 2 ^{-\frac{1}{2}}R_2(R_2 + C_2 +1) + (C_2+1)D_2 \beta_1 ^{-\frac{1}{2}}  \right]$. Thus, we arrive
	at 
	\begin{equation}
		||(\func{F}_{\ast }|_{_{\mathcal{V}_{s}}}){^{n}}||_{S^{1}\rightarrow
			S^{1}}\leq \lambda_0 ^{n}U_{1}.  \label{just}
	\end{equation}
	
	Now, recall that $\func{F}_{\ast }:S^{1}\longrightarrow S^{1}$ has an unique
	fixed point $\mu _{0}\in S^{1}$, which is a probability (see Theorem \ref%
	{gfhduer}). Consider the operator $\func{P}:S^{1}\longrightarrow \left[ \mu
	_{0}\right] $ ($\left[ \mu _{0}\right] $ is the space spanned by $\mu _{0}$%
	), defined by $\func{P}(\mu )=\mu (\Sigma )\mu _{0}$. By definition, $\func{P%
	}$ is a projection and $\dim \operatorname{Im}(\func{P})=1$. Define the operator%
	\begin{equation*}
		\func{S}:S^{1}\longrightarrow \mathcal{V}_{s},
	\end{equation*}%
	by%
	\begin{equation*}
		\func{S}(\mu )=\mu -\func{P}(\mu ),\ \ \ \mathnormal{\forall }\ \ \mu \in
		S^{1}.
	\end{equation*}%
	Thus, we set $\func{N}=\func{F}_{\ast }\circ \func{S}$ and observe that, by
	definition, $\func{P}\func{N}=\func{N}\func{P}=0$ and $\func{F}_{\ast }=%
	\func{P}+\func{N}$. Moreover, $\func{N}^{n}(\mu )=\func{F}_{\ast }{^{n}}(%
	\func{S}(\mu ))$ for all $n\geq 1$. Since $\func{S}$ is bounded and $\func{S}%
	(\mu )\in \mathcal{V}_{s}$, we get by (\ref{just}), $||\func{N}^{n}(\mu
	)||_{S^{1}}\leq \lambda_0 ^{n}U||\mu ||_{S^{1}}$, for all $n\geq 1$, where $%
	U=U_{1}||\func{S}||_{S^{1}\rightarrow S^{1}}$.
\end{proof}

\section{The Space of Bounded Variation Measures}\label{xbcvhgsafd}

In this section, we show that the disintegration of the invariant measure along $\mathcal{F}^s$ exhibits an additional degree of regularity that exceeds the mere fact that the measure belongs to the set $S^1$. In this work, we demonstrate that the disintegration of the $F$-invariant measure is of bounded variation. Studies of this kind have been conducted for other systems (see \cite{GLu}, \cite{RRR}, \cite{RR} and  \cite{DR}) and have proven fundamental, when combined with the spectral gap, for obtaining various additional properties of the dynamical system. For instance, this approach provides further insights into the set $\Theta _{\mu _0} ^1$ in equation (\ref{vbcshd}), and it has also been used to establish statistical stability in other works such as \cite{GLu} and \cite{RRRSTAB}.

\subsubsection{$\mathcal{BV}_m$-measures}
We have observed that a signed measure on $\Sigma:=I \times K$ can be disintegrated along the stable leaves $\mathcal{F}^s$ in such a way that we can view it as a family of signed measures on $I$, denoted as $\{\mu |_\gamma\}_{\gamma \in \mathcal{F}^s}$. Given that there exists a one-to-one correspondence between $\mathcal{F}^s$ and $I$, this establishment defines a path in the metric space of signed measures defined on $K$, $\mathcal{SM}(K)$, with the mapping $I \longmapsto \mathcal{SM}(K)$. In this space, $\mathcal{SM}(K)$ is equipped with the Wasserstein-Kantorovich-like metric (see Definition \ref{wasserstein}).
To make things more convenient, we employ functional notation to represent this path as $\Gamma_{\mu } : I \longrightarrow \mathcal{SM}(K)$ defined almost everywhere by $\Gamma_{\mu } (\gamma) := \mu|_\gamma$, where $(\{\mu _{\gamma }\}_{\gamma \in I},\phi_{1})$ is a disintegration of $\mu$.
However, since such a disintegration is defined for $m$-a.e $\gamma \in I$, it is crucial to note that the path $\Gamma_\mu$ is not unique. Hence, we precisely define $\Gamma_{\mu } $ as the class of almost everywhere equivalent paths that correspond to $\mu$.

\begin{definition}
	Consider a Borel signed measure $\mu \in \mathcal{AB}$ on $I \times K$ and a disintegration  $\omega=(\{\mu _{\gamma }\}_{\gamma \in I},\phi
	_1)$, where $\{\mu _{\gamma }\}_{\gamma \in I }$ is a family of
	signed measures on $K$ defined $m$-a.e. $\gamma \in I$ (and $\widehat{\mu}$-a.e. $\gamma \in I$ where $\widehat{\mu} := \pi_1{_*}\mu=\phi _1 m$) and $\phi_1$ is a marginal density defined on $I$. Denote by $\Gamma_{\mu }$ the class of equivalent paths associated to $\mu$ 
	%(of positive measures on $I$) $\Gamma_{\mu
	%}:I\longrightarrow \mathcal{SB}(I)$ defined $\mu _x$-a.e. $\gamma \in I$ by 
	\begin{equation*}
		\Gamma_{\mu }=\{ \Gamma^\omega_{\mu }\}_\omega,
	\end{equation*}
	where $\omega$ ranges on all the possible disintegrations of $\mu$ and $\Gamma^\omega_{\mu }: I\longrightarrow \mathcal{SM}(K)$ is the map associated to a given disintegration, $\omega$:
	$$\Gamma^\omega_{\mu }(\gamma )=\mu |_{\gamma } = \pi _{\gamma, 2*}\phi _1
	(\gamma)\mu _\gamma.$$We denote the $m$-full measure set on which $\Gamma_{\mu }^\omega$ is defined by $I_{\omega} \left( \subset I\right)$.
\end{definition}

%In the following, when no ambiguity is possible we will consider informally $\Gamma_{\mu }$ itself as a path.

%\footnote{Remark that to a measure many different paths and sets $I_{\Gamma_{\mu }}$ may be associated, but they coincide almost %everywhere.}. 

\begin{definition}For a given disintegration $\omega$ of $\mu$ and its functional representation $\Gamma_{\mu }^\omega$ we define the \textbf{variation of $\mu$ associated to $\omega$} by

	\begin{equation}\label{Lips1}
		V_{I_{\omega}}(\Gamma_{\mu }^\omega):= \sup _{\{x_1, x_2, \cdots, x_s\} \subset I_{\omega}} \sum _{i=1}^s || \Gamma^\omega_{\mu }(x_{i+1}) - \Gamma^\omega_{\mu }(x_{i}) ||_W.
	\end{equation} Finally, we define the \textbf{variation} of the signed measure $\mu$ by

	\begin{equation}\label{Lips2}
	V_I(\mu) :=\displaystyle{\inf_{ \Gamma_{\mu }^\omega \in \Gamma_{\mu } }\{V_{I_\omega}(\Gamma_{\mu }^\omega)\}}.
	\end{equation}
	\label{Lips3}
\end{definition}

\begin{remark}
	If $\eta \subset I_\omega$, we denote by $V_{\eta}(\Gamma_{\mu }^\omega)$ the variation of the function $\Gamma_{\mu }^\omega$ restricted to the set $\eta$. Analogously, if $\eta \subset I$ is an interval, we define $$V_{\eta}(\mu) := \displaystyle{\inf_{ \Gamma_{\mu }^\omega \in \Gamma_{\mu} } V_{\eta\cap I_\omega}(\Gamma_{\mu }^\omega)}.$$ 
\end{remark}

\begin{remark}
	When no confusion is possible, to simplify the notation, we denote $\Gamma_{\mu }^\omega (\gamma )$ just by $\mu |_{\gamma } $.
\end{remark}

\begin{definition}
	From the Definition \ref{Lips3} we define the set of the \textbf{bounded variation} measures $\mathcal{BV}_m$ as
	
	\begin{equation}
		\mathcal{BV}_m=\{\mu \in \mathcal{AB}:V_I(\mu) <+\infty \}.
	\end{equation}
\end{definition}

\subsubsection{Basic Properties of $\mathcal{BV}_m$}

\begin{lemma}
	Let $\mu \in \mathcal{BV}_m$ be a measure, and let $\Gamma^\omega_{\mu} : I_\omega \longrightarrow \mathcal{SM}(K)$ be a representative associated with a disintegration $\omega = (\{\mu_\gamma\}_{\gamma \in I_\omega}, \phi_1)$.  
	If $\mathcal{Q}$ is a partition of $I_\omega$, then  
	\[
	V_{I_\omega} (\Gamma^\omega_{\mu}) = \sum_{\eta \in \mathcal{Q}} V_\eta (\Gamma^\omega_{\mu}).
	\]  
	Consequently, if $\mathcal{Q}$ is a partition of $I$ into intervals $\eta$, then  
	\[
	V_I(\mu) = \sum_{\eta \in \mathcal{Q}} V_\eta(\mu).
	\]
\end{lemma}

\begin{proof}
	Let $\mathcal{Q} = \{\eta_1, \ldots, \eta_r\}$ be a partition of \(I_\omega\) into measurable (not necessarily connected) subsets.
	  
	Any finite sequence \(\{x_i\}_{i=0}^n \subset I_\omega\) used in the supremum defining \(V_{I_\omega}(\Gamma^\omega_\mu)\) can be rearranged (possibly subdivided) into subsequences lying entirely within each \(\eta_j\). Therefore,
	\[
	V_{I_\omega}(\Gamma^\omega_\mu) \leq \sum_{\eta \in \mathcal{Q}} V_\eta(\Gamma^\omega_\mu).
	\]
  
	To establish the reverse inequality, let \(\varepsilon > 0\) be arbitrary. For each \(\eta \in \mathcal{Q}\), choose a finite sequence \(\{x_i^\eta\}_{i=0}^{n_\eta} \subset \eta\) such that
	\[
	\sum_{i=1}^{n_\eta} \left\| \Gamma^\omega_\mu(x_i^\eta) - \Gamma^\omega_\mu(x_{i-1}^\eta) \right\|_W \geq V_\eta(\Gamma^\omega_\mu) - \varepsilon m(\eta).
	\]
	Concatenating all these sequences gives a finite sequence in \(I_\omega\), and hence
	\[
	V_{I_\omega}(\Gamma^\omega_\mu) \geq \sum_{\eta \in \mathcal{Q}} \left( V_\eta(\Gamma^\omega_\mu) - \varepsilon m(\eta) \right) = \sum_{\eta \in \mathcal{Q}} V_\eta(\Gamma^\omega_\mu) - \varepsilon.
	\]
	Since \(\varepsilon > 0\) is arbitrary, we conclude
	\[
	V_{I_\omega}(\Gamma^\omega_\mu) \geq \sum_{\eta \in \mathcal{Q}} V_\eta(\Gamma^\omega_\mu).
	\]Combining both inequalities, we get the equality:
	\[
	V_{I_\omega}(\Gamma^\omega_\mu) = \sum_{\eta \in \mathcal{Q}} V_\eta(\Gamma^\omega_\mu).
	\]For the second part, let \(\mathcal{Q}\) be a finite partition of \(I\). By definition,
	\[
	V_\eta(\mu) := \inf_{\Gamma^\omega_\mu \in \Gamma_\mu} V_{\eta \cap I_\omega}(\Gamma^\omega_\mu),
	\]
	and similarly,
	\[
	V_I(\mu) := \inf_{\Gamma^\omega_\mu \in \Gamma_\mu} V_{I_\omega}(\Gamma^\omega_\mu).
	\]
	Applying the first part of the lemma to each \(\Gamma^\omega_\mu\), we obtain:
	\[
	V_I(\mu) = \inf_{\Gamma^\omega_\mu \in \Gamma_\mu} \sum_{\eta \in \mathcal{Q}} V_{\eta \cap I_\omega}(\Gamma^\omega_\mu) \geq \sum_{\eta \in \mathcal{Q}} \inf_{\Gamma^\omega_\mu \in \Gamma_\mu} V_{\eta \cap I_\omega}(\Gamma^\omega_\mu) = \sum_{\eta \in \mathcal{Q}} V_\eta(\mu).
	\]

	In order to prove the reverse inequality, let $\Gamma^\omega_{\mu} : I_\omega \longrightarrow \mathcal{SM}(K)$ be a representative associated with a disintegration $\omega = (\{\mu_\gamma\}_{\gamma \in I_\omega}, \phi_1)$, and let $\eta \in \mathcal{Q}$. Then, 
	\[
	V_\eta(\mu) \leq V_{I_\omega \cap \eta}(\Gamma^\omega_{\mu}).
	\]
	Summing over all $\eta \in \mathcal{Q}$ and applying the first part of the lemma, we obtain
	\[
	\sum_{\eta \in \mathcal{Q}} V_\eta(\mu) \leq \sum_{\eta \in \mathcal{Q}} V_{I_\omega \cap \eta}(\Gamma^\omega_{\mu}) = V_{I_\omega}(\Gamma^\omega_{\mu}).
	\]
	Taking the infimum over all such disintegrations concludes the proof.

\end{proof}

\begin{lemma}
	Let $\mu^1, \mu^2 \in \mathcal{BV}_m$ be measures, and let $\Gamma^\omega_{\mu^i} : I_{\omega} \longrightarrow \mathcal{SM}(K)$ be representatives associated with disintegrations $\omega_i = (\{\mu^i_\gamma\}_{\gamma \in I_\omega}, \phi_1^i)$, for $i = 1, 2$. Then,
	\[
	V_{I_\omega}(\Gamma^\omega_{\mu^1} + \Gamma^\omega_{\mu^2}) \leq V_{I_\omega}(\Gamma^\omega_{\mu^1}) + V_{I_\omega}(\Gamma^\omega_{\mu^2}).
	\]
	Consequently,
	\[
	V_I(\mu^1 + \mu^2) \leq V_I(\mu^1) + V_I(\mu^2).
	\]
\end{lemma}

\begin{proof}
	Fix a finite sequence $\{\gamma_0, \gamma_1, \dots, \gamma_n\} \subset I_\omega$. Then, by the definition of variation, we have
	\begin{align*}
		\sum_{i=1}^{n} \left\| \left( \Gamma^\omega_{\mu^1} + \Gamma^\omega_{\mu^2} \right)(\gamma_i) - \left( \Gamma^\omega_{\mu^1} + \Gamma^\omega_{\mu^2} \right)(\gamma_{i-1}) \right\|_W 
		&\leq \sum_{i=1}^{n} \left\| \Gamma^\omega_{\mu^1}(\gamma_i) - \Gamma^\omega_{\mu^1}(\gamma_{i-1}) \right\|_W 
		\\&+ \sum_{i=1}^{n} \left\| \Gamma^\omega_{\mu^2}(\gamma_i) - \Gamma^\omega_{\mu^2}(\gamma_{i-1}) \right\|_W \\
		&\leq V_{I_\omega}(\Gamma^\omega_{\mu^1}) + V_{I_\omega}(\Gamma^\omega_{\mu^2}),
	\end{align*}
	where the first inequality uses the triangle inequality for the Wasserstein-type norm. Taking the supremum over all such sequences in $I_\omega$ proves the first inequality.
	
	To obtain the second inequality, note that by definition,
	\[
	V_I(\mu^1 + \mu^2) = \inf_{\Gamma^\omega_{\mu^1 + \mu^2} \in \Gamma_{\mu^1 + \mu^2}} V_{I_\omega}(\Gamma^\omega_{\mu^1 + \mu^2}).
	\]
	For any choice of representatives $\Gamma^{\omega_1}_{\mu^1} \in \Gamma_{\mu^1}$ and $\Gamma^{\omega_2}_{\mu^2} \in \Gamma_{\mu^2}$ defined on a common domain $I_\omega:= I_{\omega_1} \cap I_{\omega_2}$, the sum $\Gamma^\omega_{\mu^1} + \Gamma^\omega_{\mu^2}$ is a representative of $\mu^1 + \mu^2$. Hence,
	\[
	V_I(\mu^1 + \mu^2) \leq V_{I_\omega}(\Gamma^\omega_{\mu^1} + \Gamma^\omega_{\mu^2}) \leq V_{I_\omega}(\Gamma^\omega_{\mu^1}) + V_{I_\omega}(\Gamma^\omega_{\mu^2}) \\ \leq V_{I_{\omega_1}}(\Gamma^\omega_{\mu^1}) + V_{I_{\omega_2}}(\Gamma^\omega_{\mu^2}).
	\]
	Taking the infimum over all such representatives concludes the proof.
\end{proof}

Observe that the upcoming lemma involves a function \( \varphi \) and a map \( \Gamma: M \to \mathcal{SM}(K) \), where \( \Gamma \) is not necessarily induced by the disintegration of a measure \( \mu \). It will be employed to estimate the variation of the product of a real-valued function and the function defined in Equation~\eqref{www}. Precisely, it will be used to estimate the right-hand side of Equation~\eqref{uiyerffd} in the proof of Proposition~\ref{iuaswdas}, where we aim to control the variation of the product between the real-valued function $g$ and the map defined in Equation~\eqref{www}.

\begin{lemma}\label{bvcdgfgsd}
	Let $\Gamma : I_\omega \longrightarrow \mathcal{SM}(K)$ be a mapping and let $\varphi : I_\omega \longrightarrow \mathbb{R}$ be a function. Then, for every $\eta \subset I_\omega$, we have
	\[
	V_\eta(\varphi \Gamma) \leq \left( \esssup_{\gamma \in \eta} |\varphi(\gamma)| \right) V_\eta(\Gamma) + \left( \esssup_{\gamma \in \eta} \|\Gamma(\gamma)\|_W \right) V_\eta(\varphi),
	\]where the essential supremum is taken with respect to the measure $m$.
\end{lemma}

\begin{proof}
	Let $\{\gamma_0, \gamma_1, \dots, \gamma_n\} \subset \eta$ be an arbitrary finite sequence. Consider the sum
	\[
	\sum_{i=1}^n \left\| \varphi(\gamma_i) \Gamma(\gamma_i) - \varphi(\gamma_{i-1}) \Gamma(\gamma_{i-1}) \right\|_W.
	\]
	By adding and subtracting $\varphi(\gamma_i) \Gamma(\gamma_{i-1})$ inside each term and applying the triangle inequality, we obtain:
	\begin{align*}
		&\sum_{i=1}^n \left\| \varphi(\gamma_i) \Gamma(\gamma_i) - \varphi(\gamma_{i-1}) \Gamma(\gamma_{i-1}) \right\|_W \\
		&\leq \sum_{i=1}^n \left\| \varphi(\gamma_i) \left( \Gamma(\gamma_i) - \Gamma(\gamma_{i-1}) \right) \right\|_W 
		+ \sum_{i=1}^n \left\| \left( \varphi(\gamma_i) - \varphi(\gamma_{i-1}) \right) \Gamma(\gamma_{i-1}) \right\|_W
		\\&\leq \sum_{i=1}^n |\varphi(\gamma_i)| \cdot \left\| \Gamma(\gamma_i) - \Gamma(\gamma_{i-1}) \right\|_W 
		+ \sum_{i=1}^n |\varphi(\gamma_i) - \varphi(\gamma_{i-1})| \cdot \left\| \Gamma(\gamma_{i-1}) \right\|_W \\
		&\leq \left( \esssup_{\gamma \in \eta} |\varphi(\gamma)| \right) \sum_{i=1}^n \left\| \Gamma(\gamma_i) - \Gamma(\gamma_{i-1}) \right\|_W 
	\\	&+ \left( \esssup_{\gamma \in \eta} \|\Gamma(\gamma)\|_W \right) \sum_{i=1}^n |\varphi(\gamma_i) - \varphi(\gamma_{i-1})|.
	\end{align*}
	
	Taking the supremum over all such finite sequences in $\eta$ gives:
	\[
	V_\eta(\varphi \Gamma) \leq \left( \esssup_{\gamma \in \eta} |\varphi(\gamma)| \right) V_\eta(\Gamma) + \left( \esssup_{\gamma \in \eta} \|\Gamma(\gamma)\|_W \right) V_\eta(\varphi),
	\]
	as desired.
\end{proof}

\begin{lemma} 
	Let $\mu \in \mathcal{BV}_m$ be a measure, and let $\Gamma^\omega_{\mu} : I_\omega \longrightarrow \mathcal{SM}(K)$ be a representative associated with a disintegration $\omega = (\{\mu_\gamma\}_{\gamma \in I_\omega}, \phi_1)$. Suppose that $h : \eta \rightarrow h(\eta)$ is a homeomorphism between subintervals of $I$. Then,
	\[
	V_{\eta \cap I_\omega}(\Gamma^\omega_{\mu} \circ h) = V_{h(\eta \cap I_\omega)}(\Gamma^\omega_{\mu}).
	\]
\end{lemma}

\begin{proof}
	Fix a finite sequence $\{\gamma_0, \gamma_1, \dots, \gamma_n\} \subset \eta \cap I_\omega$. Since $h$ is a homeomorphism, the image $\{h(\gamma_0), h(\gamma_1), \dots, h(\gamma_n)\}$ is a finite sequence in $h(\eta \cap I_\omega)$, and we have
	\[
	\sum_{i=1}^n \left\| \Gamma^\omega_\mu(h(\gamma_i)) - \Gamma^\omega_\mu(h(\gamma_{i-1})) \right\|_W = \sum_{i=1}^n \left\| (\Gamma^\omega_\mu \circ h)(\gamma_i) - (\Gamma^\omega_\mu \circ h)(\gamma_{i-1}) \right\|_W.
	\]
	Taking the supremum over all such sequences in $\eta \cap I_\omega$ gives
	\[
	V_{\eta \cap I_\omega}(\Gamma^\omega_\mu \circ h) \leq V_{h(\eta \cap I_\omega)}(\Gamma^\omega_\mu).
	\]
	
	Conversely, let $\{\gamma_0', \gamma_1', \dots, \gamma_n'\} \subset h(\eta \cap I_\omega)$ be a finite sequence. Since $h$ is a homeomorphism, the inverse sequence $\{h^{-1}(\gamma_0'), \dots, h^{-1}(\gamma_n')\}$ lies in $\eta \cap I_\omega$, and we have
	\[
	\sum_{i=1}^n \left\| \Gamma^\omega_\mu(\gamma_i') - \Gamma^\omega_\mu(\gamma_{i-1}') \right\|_W = \sum_{i=1}^n \left\| (\Gamma^\omega_\mu \circ h)(h^{-1}(\gamma_i')) - (\Gamma^\omega_\mu \circ h)(h^{-1}(\gamma_{i-1}')) \right\|_W.
	\]
	Taking the supremum over all such sequences in $h(\eta \cap I_\omega)$ gives
	\[
	V_{h(\eta \cap I_\omega)}(\Gamma^\omega_\mu) \leq V_{\eta \cap I_\omega}(\Gamma^\omega_\mu \circ h).
	\]Combining both inequalities, we conclude the desired identity.
\end{proof}

\subsubsection{Properties of the action of $\func{F}_*$ on $\mathcal{BV}_m$}

In the next set of results, for a given path, $\Gamma _\mu ^\omega$, which represents the measure $\mu$, we define for each $x \in I_{\omega}\subset I$ and $\gamma = \gamma_x$, the function

\begin{equation}\label{www}
	\mu _F(x) := \func{F_\gamma }_*\mu|_\gamma,
\end{equation}where $F_\gamma :K \longrightarrow K$ is defined as

\begin{equation}\label{poier}
	F_\gamma (y) = \pi_2 \circ F \circ {(\pi _2|_\gamma)} ^{-1}(y)
\end{equation}and $\pi_2 : M\times K \longrightarrow  K$ is the projection $\pi_2(x,y)=y$.

\begin{lemma}\label{apppoas2}
		Let $F:\Sigma \longrightarrow \Sigma$ ($F=(f,G)$) be a transformation, where $f \in \mathcal{T}$ and $G$ satisfies (H1) and (H2). Then, for all positive measures $\mu \in \mathcal{BV}_m$, such that $\phi_1$ is constant $m$-a.e., it holds $$||\func{F}%
	_{x*}\mu |_x - \func{F}%
	_{y*}\mu |_y||_W \leq \alpha ||\mu |_x - \mu |_y||_W  + |G|_{\lip} d_1(x, y) ||\mu|_ y ||_W,$$ for all $x,y \in I_i \cap I_{\omega}$ and all $i=1, 2, \cdots$.
\end{lemma}

\begin{proof}

	Since $\phi_1$ is constant, we have $ (\mu|_x - \mu|_y)(K)=\phi_1(x) - \phi_1(y)=0$. Then, by Lemma \ref{quasicontract}, for all $x,y \in I_\omega \cap I_i$ it holds
	\begin{eqnarray*}
		||\func{F}%
		_{x  \ast }\mu |_{x  } - \func{F}%
		_{y \ast }\mu |_{y  }||_W &\leq & ||\func{F}%
		_{x  \ast }\mu |_{x  } - \func{F}%
		_{x \ast }\mu |_{y  }||_W + ||\func{F}%
		_{x  \ast }\mu |_{y  } - \func{F}%
		_{y \ast }\mu |_{y  }||_W
		\\&\leq & \alpha||\mu |_{x  } - \mu |_{y }||_W + ||\func{F}%
		_{x  \ast }\mu |_{y  } - \func{F}%
		_{y \ast }\mu |_{y  }||_W
		\\&\leq & \alpha ||\mu |_{x  } - \mu |_{y }||_W + ||\func{F}%
		_{x  \ast }\mu |_{y  } - \func{F}%
		_{y \ast }\mu |_{y  }||_W.
	\end{eqnarray*}Let us estimate the second summand $||\func{F}%
	_{x  \ast }\mu |_{y  } - \func{F}%
	_{y \ast }\mu |_{y  }||_W$. To do it, let $h:K \longrightarrow \mathbb{R}$ be a Lipschitz function s.t. $L(h), |h|_\infty \leq 1$. By equation (\ref{poier}) and (H2), we get 
	
	\begin{align*}
		\begin{split}
			\left|\int hd(\func{F}_{x\ast}\mu|_y)-\int hd(\func{F}_{y\ast}\mu|_y) \right|&=\left|\int\!{h(G(x,z))}d(\mu|_y)(z)\right.\\
			&\qquad\left.-\int\!{h(G(y,z))}d(\mu|_y)(z) \right|
		\end{split}
		\\&\leq\int{\left|G(x,z)-G(y,z)\right|}d(\mu|_y)(z)
		\\&\leq|G|_{\lip} d_1(x,y) \int{1}d(\mu|_y)(z) 
		\\&\leq|G|_{\lip} d_1(x,y) ||\mu|_y||_W.
	\end{align*}Thus, taking the supremum over $h$, we get 
	
	\begin{equation*}
		||\func{F}%
		_{x  \ast }\mu |_{y  } - \func{F}%
		_{y \ast }\mu |_{y  }||_W \leq|G|_{\lip} d_1(x,y) ||\mu|_y||_W.
	\end{equation*}
	%And the proof is finished.
	
\end{proof}

\begin{proposition}\label{222}
		Let $F:\Sigma \longrightarrow \Sigma$ ($F=(f,G)$) be a transformation, where $f \in \mathcal{T}$ and $G$ satisfies (H1) and (H2). Then, for all positive measures $\mu \in \mathcal{BV}_m$, such that $\phi_1$ is constant $m$-a.e., it holds $$V_{I_{\omega} \cap I_i} (\mu _F) \leq \alpha V_{I_{\omega} \cap I_i}(\Gamma^\omega _\mu)   + |G|_{\lip} \int _{I_i} {||\mu |_ y||_W}dm(y),$$ for all $i=1, 2, \cdots$.
\end{proposition}
\begin{proof}
	
	Fix $i$, and consider a finite sequence $\{x_0, \dots, x_s\} \subset I_\omega \cap I_i$. By Lemma \ref{apppoas2}, it holds $$\sum_{j=1}^s ||\mu _F (x_ i) - \mu _F(x_{i-1})||_W \leq \alpha \sum_{j=1}^s  ||\mu |_{x_i  } - \mu |_{x_{i-1}}||_W  + |G|_{\lip} d_1(x_i, x_{i-1}) ||\mu|_ {x_{i-1}} ||_W.$$We finish the proof by taking the supremum over all finite sequences of $I_\omega \cap I_i$.

\end{proof}We immediately get the following corollary.

\begin{corollary}\label{22s2}
		Let $F:\Sigma \longrightarrow \Sigma$ ($F=(f,G)$) be a transformation, where $f \in \mathcal{T}$ and $G$ satisfies (H1) and (H2). Then, for all positive measures $\mu \in \mathcal{BV}_m$, such that $\phi_1$ is constant $m$-a.e., it holds $$V_{I_{\omega}} (\mu _F) \leq \alpha V_{I_{\omega} }(\Gamma^\omega _\mu)   + |G|_{\lip} ||\mu ||_1.$$
\end{corollary}

%\begin{remark}
%	From now on, we denote the operator $\overline{\func {F}}_{\Phi,h}$ just by $\overline{\func {F}}_\Phi$. 
%\end{remark}

For the next proposition and henceforth, for a given path $\Gamma _\mu ^\omega \in \Gamma_{ \mu }$ (associated with the disintegration $\omega = (\{\mu _\gamma\}_\gamma, \phi _1)$, of $\mu$), unless written otherwise, we consider the particular path $\Gamma_{\func{F}_*\mu} ^\omega \in \Gamma_{\func{F}_* \mu}$ defined by the Proposition \ref{niceformulaab}, by the expression
\begin{equation}
	\Gamma_{\func{F}_* \mu} ^\omega (x)=\sum_{i=1}^{+\infty}{\func{F}%
		_{x _i \ast }\Gamma _\mu ^\omega (x_i)g(x_i)}\chi _{f(I_i)}(x)\ \ m\mathnormal{-a.e.}\ \ x \in I.  \label{niceformulaaareer}
\end{equation}In particular, defining
\begin{equation}
	\mu_F(x):= \func{F}%
	_{x \ast }\Gamma _\mu ^\omega (x), \label{oidfj}
\end{equation} we have
\begin{equation}
	\Gamma_{\func{F}_* \mu} ^\omega (x)=\sum_{i=1}^{+\infty}{\mu_F(x_i)g(x_i)}\chi _{f(I_i)}(x)\ \ m\mathnormal{-a.e.}\ \ x \in I.  \label{niceformulaaareer2}
\end{equation}For a given $J \in \mathcal{P}$, where $\mathcal{P}=\{I_i\}_{i=1}^\infty$ is as defined in the definition of $f$, we define the function $\func{F}_* (\mu \chi _J):I_\omega \longrightarrow \mathcal{SM}(K)$ by
\begin{equation}
		(\func{F}_* (\mu \chi _J))|_x:= \sum_{i=1}^{+\infty}{\func{F}%
			_{x _i \ast }(\Gamma _\mu ^\omega (x_i) \chi _J(x_i))g(x_i)} \chi _{f(I_i)}(x)\ \ \forall \ \ x \in I_ \omega.  \label{niceformulaaareer3}
\end{equation}Note that, it holds
\begin{equation}
\Gamma_{\func{F}_* \mu} ^\omega =\sum_{J \in \mathcal{P}}\func{F}_* (\mu \chi _J). \label{jhfgjge}
\end{equation}For a given $J\in \mathcal{P}$, we denote $J_{\omega}:= J \cap I_ \omega$ and $f_J:= f|_J$. Therefore, for all $y=f_J(x)$ where $x \in J_{\omega}$, it holds (by (\ref{niceformulaaareer3}))

\begin{eqnarray*}
\func{F}_{*} (\mu \chi _J)|_y&:=& \sum_{i=1}^{+\infty}{\func{F}%
	_{x _i \ast }\mu |_{x_i} \chi _J(x_i))g(x_i)} \chi _{f(I_i)}(y) \\&=&  {\func{F}%
	_{x \ast }\mu |_{x}g(x) \chi _J(x)} \chi _{f(J)}(y) \\&=&  {\func{F}%
	_{x \ast }\mu |_{x} g(x) \chi _J(x)}\\&=&  \mu _F(x) g(x)\chi _J(x).
\end{eqnarray*} Thus,

\begin{eqnarray*}
	V_{I_{\omega}} (\func{F}_{*} (\mu \chi _J))&=& V_{I_{\omega}}(\mu _F g\chi _J) \\&=& V_J(\mu _F g).
\end{eqnarray*} Summing the above relation we get

\begin{equation}\label{ieduriwre}
	\sum _{J \in \mathcal{P}} V_{I_{\omega}} (\func{F}_{*} (\mu \chi _J))=\sum _{J \in \mathcal{P}} V_{J_\omega}(\mu _F g) = V_{I{_\omega}}(\mu _F g).
\end{equation}By equations (\ref{jhfgjge}) and (\ref{ieduriwre}), it holds %\marginpar{a equação faltando aqui é a desigualdade triangular para variação}
\begin{eqnarray*}\label{ieduriwre2}
	V _{I_\omega}(\Gamma_{\func{F}_* \mu} ^\omega ) &\leq& \sum_{J \in \mathcal{P}} V _{I_\omega}(\func{F}_* (\mu \chi _J) ) \\ &=&\sum _{J \in \mathcal{P}} V_{J_\omega}(\mu _F g) \\&=& V_{I{_\omega}}(\mu _F g).
\end{eqnarray*}Thus,
\begin{equation}\label{uiyerffd}
	V _{I_\omega}(\Gamma_{\func{F}_* \mu} ^\omega ) \leq V_{I{_\omega}}(\mu _F g).
\end{equation}Equation (\ref{uiyerffd}) and Corollary \ref{22s2} yield the following proposition.

%Recall that $\Gamma_{\mu} ^\omega (\gamma) = \mu|_\gamma:= \pi_{2*}(\phi_{1}(\gamma)\mu _\gamma)$ and in particular $\Gamma_{\func{\overline{F}}_\Phi\mu} ^\omega (\gamma) = (\func{\overline{F}}_\Phi\mu)|_\gamma = \pi_{2*}(\mathcal{\overline{L}}_\varphi\phi_1(\gamma)(\func{F}_\Phi\mu )_\gamma)$, where $\phi_1 = \dfrac{d \pi _{1*} \mu}{dm}$. 
\begin{proposition}\label{iuaswdas}
		Let $F:\Sigma \longrightarrow \Sigma$ ($F=(f,G)$) be a transformation, where $f \in \mathcal{T}$ and $G$ satisfies (H1) and (H2). Then, for all positive measures $\mu \in \mathcal{BV}_m$, such that $\phi_1$ is constant $m$-a.e., it holds 
		\begin{equation}\label{uieriea}
		V(\Gamma_{\func{F}_*} ^\omega\mu)  \leq \alpha _3 V(\Gamma_{\mu}^\omega) + U_3||\mu||_1
		\end{equation}where the essential supremum is taken with respect to the measure $m$, $\alpha_3=\alpha \esssup g$ and $U_3 = |G|_{\lip} \cdot \esssup g  + V_{I{_\omega}}(g)$.
\end{proposition}

\begin{proof}  
	
	Since $\phi_1$ is constant $m$-a.e., there exists a constant $c \geq 0$ such that $|\phi_1|_1=|\phi_1|_\infty=c$. Furthermore, since by Lemma \ref{niceformulaac} all probabilities $\nu$ satisfies $||\nu||_W=1$ and $\mu$ is a positive measure, we have 
	$$||\mu||_1 = \int ||\mu|_\gamma||_W dm(\gamma)= \int|\phi_1|dm(\gamma) =c.$$By Lemma \ref{niceformulaac}, we also have that
	
	\begin{eqnarray*}
		||\mu_F(x)||_W&=& ||\func{F}
		_{x \ast }\Gamma _\mu ^\omega (x)||_W \\&\leq& || \Gamma _\mu ^\omega (x)||_W\\&=& |\phi_1(x)| \\&=&c.
	\end{eqnarray*}Thus,
	
	\begin{equation}\label{tyurt}
		\esssup  ||\mu_F||_W \leq ||\mu||_1,
	\end{equation}where the essential supremum is taken with respect to $m$.
	
	By equations (\ref{uiyerffd}) and (\ref{tyurt}), Lemma \ref{bvcdgfgsd} and Corollary \ref{22s2}, we have
	
	\begin{eqnarray*}
		V _{I_\omega}(\Gamma_{\func{F}_* \mu} ^\omega ) 
		&\leq& V_{I{_\omega}}(\mu _F g)
		\\&\leq& V_{I{_\omega}}(\mu _F)\esssup g + V_{I{_\omega}}(g) \esssup ||\mu _F||_W
\\&\leq& (\alpha \esssup g )V_{I_{\omega} }(\Gamma^\omega _\mu) + |G|_{\lip} ||\mu ||_1 \esssup g \\ &+& V_{I{_\omega}}(g) \esssup ||\mu _F||_W
\\&\leq& (\alpha \esssup g )V_{I_{\omega} }(\Gamma^\omega _\mu) + (\esssup g |G|_{\lip} + V_{I{_\omega}}(g)) ||\mu ||_1.  
\end{eqnarray*}

\end{proof}

\begin{corollary}\label{iuaswdas1}
	Let $F:\Sigma \longrightarrow \Sigma$ ($F=(f,G)$) be a transformation, where $f \in \mathcal{T}$ and $F$ satisfies (H3). Then, for all positive measures $\mu \in \mathcal{BV}_m$, such that $\phi_1$ is constant $m$-a.e., it holds ($\overline{F}:=F^k$) 
	\begin{equation}\label{uieriea}
		V(\Gamma_{\func{\overline{F}}_*} ^\omega\mu)  \leq \alpha _4 V(\Gamma_{\mu}^\omega) + U_4||\mu||_1
	\end{equation}where $\alpha_4:=\alpha^k \esssup \frac{1}{|(f^k)'|}$ and $U_4 = |G|_{\lip} \cdot \esssup \frac{1}{|(f^k)'|}  + V_{I{_\omega}}(\frac{1}{|(f^k)'|})$ and the essential supremum is taken with respect to the measure $m$.
\end{corollary}

By iterating the inequality (\ref{uieriea}) obtained in Corollary \ref{iuaswdas}, along with a standard computation, we arrive at the following result, the proof of which is omitted.

\begin{corollary}\label{kjdfhkkhfdjfht}
		Let $F:\Sigma \longrightarrow \Sigma$ ($F=(f,G)$) be a transformation, where $f \in \mathcal{T}$ and $F$ satisfies (H3). Then, for all positive measures $\mu \in \mathcal{BV}_m$, such that $\phi_1$ is constant $m$-a.e., it holds ($\overline{F}:=F^k$) 
	\begin{equation}\label{erkjwr166}
	V(\Gamma_{\func{\overline{F}}_*^n \mu} ^\omega)  \leq \alpha_4^n V(\Gamma_{\mu}^\omega) + \dfrac{U_4}{1-\alpha_4}||\mu||_1,
	\end{equation}
	for all $n\geq 1$.
\end{corollary}

\begin{remark}\label{kjedhkfjhksjdf}
	Taking the infimum over all $\Gamma_{ \mu } ^\omega  \in \Gamma_{ \mu }$ and all $\Gamma_{\func{F}_*^n\mu}^\omega  \in \Gamma_{\func{F}_* ^n}\mu$ on both sides of inequality (\ref{erkjwr166}), we get 
	
	\begin{equation}\label{fljghlfjdgkdg}
		V(\func{\overline{F}}_* ^n\mu)  \leq \alpha_4^n V(\mu) + \dfrac{U_4}{1-\alpha_4}||\mu||_1, 
	\end{equation}for all positive measures $\mu \in \mathcal{BV}_m$ such that $\phi_1$ is constant $m$-almost everywhere.
	
	The above equation (\ref{fljghlfjdgkdg}) will give a uniform bound (see the proof of Theorem \ref{kjdfhkkhfdjfh}) for the variation of the measure $\func{F}_*^{n} \nu$, for all $n$. Where $\nu$ is defined as the product $\nu =m \times \nu_2$, for a fixed probability measure $\nu_2$ on $K$ (see the following Remark \ref{riirorpdf}). The uniform bound will be useful later on (see Theorem \ref{disisisi}).
	
\end{remark}

\begin{remark}\label{riirorpdf}
	Consider the probability measure $\nu$ defined in Remark \ref{kjedhkfjhksjdf}, i.e., $\nu=m \times \nu_2$, where $\nu_2$ is a given probability measure on $K$ and $m$ is the Lebesgue measure on $I$. Besides that, consider its trivial disintegration $\omega_0 =(\{\nu_{ \gamma}  \}_{\gamma}, \phi_1)$, given by $\nu_{ \gamma} = \func{\pi _{2,\gamma}^{-1}{_*}}\nu_2$, for all $\gamma$ and $\phi _1 \equiv 1$. According to this definition, it holds that 
	\begin{equation*}
		\nu|_\gamma = \nu_2, \ \ \forall \ \gamma.
	\end{equation*}In other words, the path $\Gamma ^{\omega _0}_{\nu}$ is constant: $\Gamma ^{\omega _0}_{\nu} (\gamma)= \nu_2$ for all $\gamma$. Hence, $\nu \in \mathcal{BV}_m$.  Moreover, for each $n \in \mathbb{N}$, let $\omega_n$ be the particular disintegration of the measure $\func{\overline{F}}_*^n\nu$ defined from $\omega_0$ as an application of Proposition \ref{niceformulaab} and, by a simple induction, consider the path $\Gamma^{\omega_{n}}_{\func{\overline{F}}_*^n \nu}$ associated with $\omega _n$. This path will be used in the proof of the next proposition.

\end{remark}

For the next result, recall that by Theorem \ref{gfhduer} a map $F = (f,G)$, where $f \in \mathcal{T}$ and $G$ satisfies (H1), has a unique invariant measure $\mu _{0}\in S^{1}$. We will show that $\mu_0$ admits a regular disintegration, which implies that $\mu _0$ belongs to $\mathcal{BV}_m$. Similar results for other classes of systems can be found in \cite{GLu}, \cite{RRR}, and \cite{DR}. This property will then be used to establish the exponential decay of correlations for Lipschitz functions.

\begin{athm}\label{kjdfhkkhfdjfh}
	Let $F:\Sigma \longrightarrow \Sigma$ ($F=(f,G)$) be a transformation, where $f \in \mathcal{T}$ and $F$ satisfies (H3). Let $\mu_0$ be its unique invariant measure in $S^1$. Then, $\mu_0 \in \mathcal{BV}_m$ and it holds 
	\begin{equation}\label{erkjwr}
		V(\mu_0)  \leq  \dfrac{U_4}{1-\alpha_4},
	\end{equation}where $\alpha_4:=\alpha^k \esssup \frac{1}{|(f^k)'|}$ and $U_4 = |G|_{\lip} \cdot \esssup \frac{1}{|(f^k)'|}  + V_{I{_\omega}}(\frac{1}{|(f^k)'|})$.
\end{athm}

\begin{proof}

	Consider the path $\Gamma^{\omega_n}_{\func{\overline{F}}_*^n}\nu$, defined in Remark \ref{riirorpdf}, which represents the measure $\func{\overline{F}}_*^n\nu$, where $\overline{F}:=F^k$.

	%First of all, it is not hard to prove that if $\Gamma^\omega_{n}:\widehat{I}\longrightarrow \mathcal{SB}([0,1])$ is a sequence of paths which converges
	%to $\Gamma_{\mu _{0}}^\omega:\widehat{I}\longrightarrow \mathcal{SB}([0,1])$ pointwise on a
	%full measure set $\widehat{I}\subset I$, then for every fixed partition $\mathcal{P%
		%}=\{x_{0},\cdots ,x_{n}\}\subset \widehat{I}$ it holds 
	%\begin{equation*}
	%\lim_{n\longrightarrow \infty }{\Var(\Gamma^\omega_{{n}},\mathcal{P})}=\Var(\Gamma^\omega_{\mu _{0}},%
	%\mathcal{P}).
	%\end{equation*}
	
	According to Theorem \ref{gfhduer}, let $\mu _{0}$ be the unique $F$-invariant probability measure in $S^1$. It holds that $\func{\overline{F}}_* \mu_0 = \mu_0$. Consider the measure $\nu$, defined in Remark \ref{riirorpdf} and its iterates $\func{\overline{F}}_*^n(\nu)$. By Theorem \ref{spgap}, this sequence converges to $\mu _{0}$
	in $\mathcal{L}^{1}$. It implies that a subsequence of $\{\Gamma_{\func{\overline{F}{_\ast }}^n(\nu)} ^{\omega _n}\}_{n}$ converges $m$-a.e. to $\Gamma_{\mu _{0}}^\omega\in \Gamma_{\mu_0 }$ (in $\mathcal{SB}(K)$ with respect to the metric defined in Definition \ref{wasserstein}), where $\Gamma_{\mu _{0}}^\omega$ is a path given by the Rokhlin Disintegration Theorem and $\{\Gamma_{\func{\overline{F}}_*^n(\nu)} ^{\omega_n}\}_{n}$ is defined in Remark (\ref{riirorpdf}).  Let us suppose that $\{\Gamma_{\func{\overline{F}{_\ast }}^n(\nu)} ^{\omega _n}\}_{n}$ is such a convergent subsequence. It implies that $\{\Gamma_{\func{\overline{F}}_*^n(\nu)} ^{\omega_n}\}_{n}$ converges pointwise to $\Gamma_{\mu _{0}}^\omega$ on a full measure set $I_{\omega}\subset I$. To simplify the notation, let us denote $%
	\Gamma_{n}:=\Gamma^{\omega_n}_{\func{\overline{F}}_*^n(\nu)}|_{%
		I_{\omega}}$ and $\Gamma:=\Gamma^\omega _{\mu _{0}}|_{I_{\omega}}$. Since $\{\Gamma_{n} \}_n $ converges pointwise to $\Gamma$, it holds $V_{I_{\omega}}(\Gamma_{n}) \longrightarrow V_{I_{\omega}}(\Gamma)$ as $n \rightarrow \infty$. Indeed, consider $\{x_0, x_1, \cdots, x_s \} \subset I_{\omega}$. Then,
	
	\begin{eqnarray*}
		\lim _{n \longrightarrow \infty} \sum_{i=1}^s{||\Gamma_n (x_i) - \Gamma _n(x_{i-1})||_W} &= & \sum_{i=1}^s{||\Gamma (x_i) - \Gamma (x_{i-1})||_W}.
	\end{eqnarray*} On the other hand, by equation (\ref{erkjwr166}) of Corollary \ref{kjdfhkkhfdjfht}, the argument of the left hand side is bounded by $\dfrac{U_4}{1-\alpha_4}$ for all $n\geq 1$, since $V(\Gamma)=0$ and $||\nu||_1=1$. Then, 
	\begin{eqnarray*}
		\sum_{i=1}^s{||\Gamma (x_i) - \Gamma (x_{i-1})||_W} &\leq & \dfrac{U_4}{1-\alpha_4}.
	\end{eqnarray*} Thus, $V_{I_\omega}(\Gamma^\omega_{\mu _0})\leq\dfrac{U_4}{1-\alpha_4}$ and taking the infimum over $\Gamma _ {\mu_0}$ we get $V_I(\mu _0) \leq\dfrac{U_4}{1-\alpha_4}$.
\end{proof}

\section{Exponential decay of correlations}\label{dfjgsghdfjasdf}

In this section, we will show how Theorem \ref{spgap} implies an exponential rate of convergence for the limit $$\lim_{n \to \infty} {C_n(u_1,u_2)}=0,$$where 

\begin{equation*}\label{adsfg}
C_n(u_1,u_2):=\left| \int{(u_1 \circ F^n )u_2}d\mu_0 - \int{u_1}d\mu_0\int{u_2}d\mu_0\right|,
\end{equation*} $u_1: \Sigma \longrightarrow \mathbb{R} $ is a Lipschitz function, and $u_2 \in \Theta _{\mu _0} ^1$. The set $\Theta _{\mu _0} ^1$ is defined as 
\begin{equation}\label{vbcshd}
	\Theta _{\mu _0} ^1:= \{ u: \Sigma \longrightarrow \mathbb{R}: u\mu_0 \in S^1\},
\end{equation} and the measure $u\mu_0$ is defined by $u\mu_0(E):=\int _E{u}d\mu_0$ for all measurable set $E$.

\begin{proposition}\label{hjgsdfsa}
Let $F:\Sigma \longrightarrow \Sigma$ ($F=(f,G)$) be a transformation, where $f \in \mathcal{T}$ and $G$ satisfies (H1). For all Lipschitz function $u_1: \Sigma \longrightarrow \mathbb{R} $ and all $u_2 \in \Theta _{\mu _0} ^1$, it holds $$\left| \int{(u_1 \circ F^n )  u_2}d\mu_0 - \int{u_1}d\mu_0 \int{u_2}d\mu_0 \right| \leq ||u_2 \mu _0||_{S^{1}} U |u_1|_{\lip}  \lambda_0 ^{n} \ \ \forall n \geq 1,$$where $\lambda_0$ and $U$ are from Theorem \ref{spgap} and $|u_1|_{\lip} := |u_1|_\infty + L(u_1)$. 
\end{proposition}

\begin{proof}
	
	Let $u_1: \Sigma \longrightarrow \mathbb{R} $ be a Lipschitz function, and $u_2 \in \Theta _{\mu _0} ^1$. By Theorem \ref{spgap}, we have
	
	\begin{eqnarray*}
		\left| \int{(u_1 \circ F^n )  u_2}d\mu_0 - \int{u_1}d\mu_0 \int{u_2}d\mu_0 \right| 
		&=& \left| \int{u_1}d \func{F^*}{^n} (u_2\mu_0) - \int{u_1}d\func{P}(u_2\mu_0) \right|
		\\&\leq& \left|\left|  \func{F^*}{^n} (u_2\mu_0) - \func{P}(u_2\mu_0) \right|\right|_W |u_1|_{\lip} 
		\\&=& \left|\left|  \func{N}{^n}(u_2\mu_0) \right|\right|_W |u_1|_{\lip} 
		\\&\leq & \left|\left|  \func{N}{^n}(u_2\mu_0) \right|\right|_{S^1} |u_1|_{\lip} 
		\\&\leq & ||u_2 \mu _0||_{S^{1}} U |u_1|_{\lip}  \lambda _0 ^{n}.
	\end{eqnarray*}
	
\end{proof}

The proof of the next lemma is analogous to the proof of Lemma 8.1 of \cite{RRR}, thus we omit it here.

\begin{lemma}\label{hdgfghddsfg}
	Let $(\{\mu_{0, \gamma}\}_\gamma, \phi_1)$ be the disintegration of $\mu _0$, along the partition $\mathcal{F}^s:=\{\{\gamma\} \times K: \gamma \in I\}$, and for a $\mu_0$ integrable function $h:I \times K \longrightarrow \mathbb{R}$, denote by $\nu$ the measure $\nu:=h\mu_0$ ($ h\mu_0(E) := \int _E {h}d\mu _0$). If $(\{\nu_{ \gamma}\}_\gamma, \widehat{\nu} )$ is the disintegration of $\nu$, where $\widehat{\nu}:=\pi_1{_*} \nu$, then it holds $\widehat{\nu} \ll m$ and $\nu _\gamma \ll \mu_{0, \gamma}$. Moreover, denoting $\overline{h}:=\dfrac{d\widehat{\nu}}{dm}$, it holds 
	\begin{equation}\label{fjgh}
		\overline{h}(\gamma)=\int_{K}{h(\gamma, y)}d(\mu_0|_\gamma)(y),
	\end{equation} and for $\widehat{\nu}$-a.e. $\gamma \in I$

	\begin{equation}\label{gdfgdgf}
		\dfrac{d\nu _{ \gamma}}{d\mu _{0, \gamma}}(y) =
		\begin{cases}
			\dfrac{h|_\gamma (y)}{\int{h|_\gamma(y)}d\mu_{0,\gamma}(y)} , \ \hbox{if} \ \gamma \in B ^c \\
			0, \ \hbox{if} \ \gamma \in B,
		\end{cases} \hbox{for all} \ \ y \in K,
	\end{equation}where $B :=  \overline{h} ^{-1}(0)$.
\end{lemma}

\begin{athm}\label{disisisi}
	Let $F:\Sigma \longrightarrow \Sigma$ ($F=(f,G)$) be a transformation, where $f \in \mathcal{T}$ and $F$ satisfies (H3). Let $\mu_0 \in S^1$ be the unique $F$-invariant measure in $S^1$. Then, $\lip(\Sigma) \subset \Theta _{\mu_0} ^1$, where $\lip(\Sigma)$ is set of real Lipschitz functions defined on $\Sigma$. In particular, for all Lipschitz functions $u_1:\Sigma \longrightarrow \mathcal{R}$ and $u_2: \Sigma \longrightarrow \mathcal{R}$ it holds 
	
	\begin{equation}\label{uyrtetre}
	C_n(u_1,u_2) \leq ||u_1 \mu _0||_{S^{1}} U |u_2|_{\lip}  \lambda_0 ^{n} \ \ \forall n \geq 1,
	\end{equation}where $\lambda_0$ and $U$ are from Proposition \ref{hjgsdfsa} and $|u|_{\lip} := |u|_\infty + L(u)$. 
\end{athm}

\begin{proof}
	Let $(\{\mu_{0, \gamma}\}_\gamma, \phi_1)$ be the disintegration of $\mu _0$ and denote by $\nu$ the measure $\nu:=h_1\mu_0$ ($ u_1\mu_0(E) := \int _E {u_1}d\mu _0$). If $(\{\nu_{ \gamma}\}_\gamma, \widehat{\nu} )$ is the disintegration of $\nu$, then by Lemma \ref{hdgfghddsfg} it holds $\widehat{\nu} \ll m$ and $\nu _\gamma \ll \mu_{0, \gamma}$. Moreover, denoting $\overline{u}_1:=\dfrac{d\widehat{\nu}}{dm}$, it holds 
	\begin{equation*}
		\overline{u}_1(\gamma)=\int_{K}{u_1(\gamma, y)}d(\mu_0|_\gamma),
	\end{equation*} and

	\begin{equation*}
		\dfrac{d\nu _\gamma}{d\mu_{0, \gamma}} (y) = 
		\begin{cases}
			\dfrac{u_1(\gamma,y) }{\overline{u}_1(\gamma)}, & \text{if } \overline{u}_1(\gamma) \neq 0 \\
			0, & \text{if } \overline{u}_1(\gamma) = 0.
		\end{cases}
	\end{equation*}It is immediate that $\nu \in \mathcal{L}^1$. 
	To verify that $\overline{u}_1 \in BV_m$, we estimate the total variation of $\overline{u}_1$. Let ${0 = \gamma_0 < \gamma_1 < \cdots < \gamma_n = 1}$ be an arbitrary finite sequence in the domain of a representative of $\overline{u}_1$. Then we have:

	\begin{eqnarray*}
		|\overline{u}_1(\gamma_{i}) - \overline{u}_1(\gamma_{i-1})| &\leq& \left|\int_{K}{u_1(\gamma _i, y)}d(\mu_0|_{\gamma_i}) - \int_{K}{u_1(\gamma _{i-1}, y)}d(\mu_0|_{\gamma_{i-1}}) \right| 
		\\&\leq& \left|\int_{K}{u_1(\gamma _i, y)}d(\mu_0|_{\gamma_i}) - \int_{K}{u_1(\gamma _{i}, y)}d(\mu_0|_{\gamma_{i-1}}) \right| 
		\\&+& \left|\int_{K}{u_1(\gamma _i, y)}d(\mu_0|_{\gamma_{i-1}}) - \int_{K}{u_1(\gamma _{i-1}, y)}d(\mu_0|_{\gamma_{i-1}}) \right| 
		\\&\leq& \left|\int_{K}{u_1(\gamma _i, y)}d(\mu_0|_{\gamma_i}-\mu_0|_{\gamma_{i-1}}) \right| 
		\\&+& \left|\int_{K}{u_1(\gamma _i, y)-u_1(\gamma _{i-1}, y)}d(\mu_0|_{\gamma_{i-1}}) \right|
		\\&\leq& ||u_1||_{\lip} ||\mu_0|_{\gamma_i}-\mu_0|_{\gamma_{i-1}}||_W + L(u_1)|\gamma_i - \gamma_{i-1}| \left|\phi _1 \right|_{\infty}.
	\end{eqnarray*}Thus, $V{(\overline{u}_1)} < \infty$ and $\overline{u}_1 \in BV_m$. By Proposition \ref{hjgsdfsa}, we conclude the proof of the theorem. 
\end{proof}

\end{document}